\PassOptionsToPackage{pagewise}{lineno}
\documentclass[review]{elsarticle}

\usepackage{lineno,hyperref}

\journal{Transportation Research Part C}
\usepackage{multirow}
\usepackage{etex}
\usepackage{amsfonts}
\usepackage{mathrsfs}
\usepackage{comment}
\usepackage{tablefootnote}
\usepackage{float}
\usepackage{algorithm}
\usepackage{algpseudocode}
\usepackage{todonotes}
\usepackage{amsmath}
\usepackage{longtable}
\usepackage{threeparttable}
\usepackage{caption}
\usepackage{pdflscape}
\usepackage{graphics}
\usepackage{subfigure}
\usepackage{subcaption}
\usepackage{ulem}
\usepackage{soul}
\usepackage{adjustbox}
\usepackage[table]{xcolor}
\definecolor{cT}{HTML}{0E8C8C} \definecolor{cR}{HTML}{C0392B}
\definecolor{cOpt}{HTML}{FFFFFF}

\definecolor{shade}{HTML}{808080}

\usepackage{xcolor}

\usepackage{xspace}
\newcommand\BRFrag{battery-restricted fragment\xspace}
\newcommand\BRFrags{battery-restricted fragments\xspace}
\newcommand\Must{must-dispatch\xspace}
\newcommand\Postpone{postponable\xspace}
\newcommand\state{\boldsymbol{x}_e\xspace}
\newcommand\frag{Frag\xspace}
\newcommand\feature{\phi\xspace}

\newcommand\lb{anticipative reference\xspace}

\usepackage{microtype}
\usepackage{amsthm}
\usepackage{amsmath}

\usepackage{amssymb}
\usepackage{stmaryrd}
\usepackage{color,xcolor}
\usepackage{soul}
\soulregister{\cite}7
\soulregister{\citep}7
\soulregister{\ref}7

\usepackage{graphicx}
\usepackage{geometry}
\usepackage{tikz}
\usetikzlibrary{shapes.geometric, arrows}
\usepackage{diagbox}

\usepackage{footnote}
\makesavenoteenv{table}
\newcommand{\beq}{\begin{equation}}
\newcommand{\eeq}{\end{equation}}

\theoremstyle{definition}
\newtheorem{exmp}{Example}
\theoremstyle{remark}

\newtheorem{defn}{Definition}
\newtheorem{proposition}{Proposition}

\makeatletter 
\def\@makecaption#1#2{%
  \vskip\abovecaptionskip
  \sbox\@tempboxa{#1 #2}%
    {\bfseries #1} #2\par
  \vskip\belowcaptionskip}
\makeatother
\usepackage{booktabs}
\usepackage[justification=centering]{caption}

\usepackage{hyperref}
\hypersetup{
    colorlinks=true,
    linkcolor=blue,
    filecolor=gray,      
    urlcolor=blue,
    citecolor=blue,
}
\biboptions{authoryear}

\allowdisplaybreaks

\usepackage{comment}

\begin{document}




\begin{frontmatter}
\title{Combinatorial Optimization Augmented Machine Learning for Dynamic Electric Autonomous Dial-a-Ride Problem}
\author[address1]{Yue Su}
\ead{yue.su@inria.fr}
\author[address2]{Léa Ricard}
\ead{lea.ricard@epfl.ch}

\address[address1]{Centre INRIA de l’Université de Lille, 59000 Lille, France.}
\address[address2]{Transport and Mobility Laboratory, School of Architecture, Civil and Environmental Engineering, Ecole Polytechnique F\'ed\'erale de Lausanne, Switzerland}

\cortext[correspondingauthor]{Corresponding author: lea.ricard@epfl.ch}
\begin{abstract}
This study introduces a decision-epoch-based dynamic electric autonomous dial-a-ride problem (Dyn-EADARP), in which incoming requests are collected and processed at periodic decision epochs. A key decision is not only how to serve requests, but also when to dispatch them. At each decision epoch, the service provider decides which requests to serve and which to postpone, while jointly determining vehicle routes, schedules, and charging decisions. Unexecuted parts of existing plans can be revised as new information becomes available. To solve this problem, we develop an ML--CO policy following the combinatorial optimization augmented machine learning (COAML) framework, which combines a statistical model with a combinatorial optimization layer for decision making. The statistical model predicts prizes for available requests, and a prize-collecting E-ADARP uses these prizes to jointly determine request selection, routing, scheduling, and charging. The statistical model is trained directly to improve the decisions produced by the optimization layer. Computational experiments on 320 test instances demonstrate the efficiency of ML--CO, which solves instances with nearly 500 requests in about 6 seconds on average. It achieves 8.7\%--14.3\% lower objective values than benchmark policies and an average gap of 4.8\% to the anticipative reference. The results provide several managerial insights. First, serving requests immediately is not always best, as selectively postponing some requests can create better ride-sharing opportunities. Second, revising existing plans preserves operational flexibility and substantially improves solution quality. Finally, more frequent decision making does not necessarily improve performance, highlighting the importance of choosing an appropriate decision frequency.

\end{abstract}


\end{frontmatter}


\section{Introduction}
Rapid urbanization has increased urban travel demand and intensified traffic congestion, energy consumption, and air pollution \citep{levy2010evaluation}. More than half of the world's population currently lives in urban areas, and this share is expected to reach 68\% by 2050 \citep{un2019worldurbanization}. These trends create a growing need for efficient and environmentally sustainable urban mobility services.

Demand-responsive mobility services provide a flexible way to meet this need. In particular, ride-sharing can serve multiple requests in the same vehicle, improving vehicle utilization while maintaining an acceptable level of service \citep{jin2018ridesourcing}. The underlying optimization problem is usually formulated as a dial-a-ride problem (DARP), which determines vehicle routes and schedules to provide ride-sharing services for requests with different origins and destinations \citep{cordeau2007dial}.
At the same time, electric vehicles and autonomous driving offer new opportunities to reduce emissions and support more flexible operations. Combining these technologies with ride-sharing gives rise to the electric autonomous dial-a-ride problem (E-ADARP), which is introduced in \cite{bongiovanni2019electric}. In addition to the classical DARP constraints, the E-ADARP considers battery consumption and charging operations. Decisions on charging-station visits and recharging time must therefore be made jointly with routing and scheduling decisions, making the problem more complex than the classical DARP \citep{bongiovanni2024ride}.
Most works of the E-ADARP solve the static version in which customer requests are known in advance \citep{su2023deterministic,su2024branch,su2025bi,limmer2023bilevel,stallhofer2025event}. However, in practice, customer requests usually arrive in real time, requiring vehicle routes and service decisions to be updated dynamically. 
More recently, several works \citep{bongiovanni2022machine,huang2025genetic,Tomandl2026} have extended the static E-ADARP to an ``event-driven'' dynamic setting, where a reoptimization process is invoked whenever a new request arrives. 
A common limitation of these approaches is that the reoptimization frequency is directly tied to the request-arrival frequency: each request arrival requires solving a new static E-ADARP. In high-demand settings, repeatedly solving these problems can become computationally demanding and is often operationally unnecessary, limiting the practical applicability of these approaches.
An alternative is to collect incoming requests and process them at periodic decision epochs, thereby separating request arrivals from dispatching decisions \citep{alonso2017demand}. At each decision epoch, available requests may either be dispatched immediately or postponed to a later epoch, making dispatch timing an explicit decision. Such a setting has been studied in dynamic ride-pooling \citep{Fielbaum2022}, but not yet in the E-ADARP.


Motivated by these observations, we study a new variant of the E-ADARP, where request arrivals between two consecutive epochs are collected and become available for decision making at the next epoch. At each epoch, the service provider decides which available requests to dispatch and which to postpone, while jointly determining vehicle routes and schedules. Previously planned but unexecuted portions of routes can be revised as new information becomes available. To solve this problem efficiently, we develop a highly efficient policy based on a combinatorial optimization augmented machine learning (COAML) framework. The COAML architecture consists of an upstream statistical model and a downstream optimization layer \citep{dalle2022learning}, with the statistical model trained directly to improve the quality of the decisions produced by the optimization layer \citep{elmachtoub2020decision,schiffer2026combinatorial}. The main contributions of this work are summarized as follows:
\begin{itemize}
\item We introduce a decision-epoch-based dynamic E-ADARP (Dyn-EADARP) that decouples reoptimization from individual request arrivals and makes dispatch timing an explicit decision. This allows requests to be strategically postponed and grouped with future arrivals for better ride-sharing opportunities, while avoiding unnecessary request-by-request reoptimization.

\item We develop an ML--CO policy for the Dyn-EADARP based on the COAML framework of \cite{baty2024combinatorial}, with substantial adaptations to accommodate the more complex routing and operational constraints of the E-ADARP. In particular, we design a problem-tailored combinatorial optimization (CO) oracle based on an efficient heuristic, which enables the learning framework to jointly handle the request selection, routing, scheduling, and recharging decisions arising at each epoch.

\item Extensive numerical experiments demonstrate the effectiveness of the proposed policy. The ML--CO policy solves Dyn-EADARP instances with nearly 500 requests in about 6 seconds on average. Compared with standard benchmark policies, the ML--CO policy reduces the objective value by 8.7\%--14.3\% and achieves an average gap of 4.8\% to the \lb.
The experiments also provide several managerial insights. The first insight is the value of strategic postponement. We show that strategically postponing some requests can create better ride-sharing opportunities with future requests, leading to lower routing costs. The second insight is the value of route revision. Allowing previously planned routes to be revised can substantially improve solution quality and reduce the number of dispatched vehicles. Finally, more frequent decision making does not necessarily lead to better performance.

\end{itemize}

The remainder of the paper is organized as follows. Section~\ref{sec:literature review} reviews the related literature on the dynamic DARP, E-ADARP and the COAML. Section~\ref{sec:DynEADARP} describes the problem setting and formulates the Dyn-EADARP as a Markov Decision Process. Section~\ref{sec:policy_architecture} introduces the policy architecture based on the COAML framework. Section~\ref{sec:learning_algorithm} presents the learning algorithm, including the loss function and the anticipative oracle used to generate training data. Section~\ref{sec:numerical_results} describes the dynamic instance generation and benchmark policies, and presents the numerical results. Section~\ref{sec:conclusion} concludes the paper and discusses future extensions.

\section{Literature review} \label{sec:literature review}

In this section, we review related works along two directions relevant to this paper: works on the dynamic DARP and E-ADARP (Section~\ref{sec: review of dynamic VRP}), and works that employ combinatorial optimization augmented machine learning (COAML) (Section~\ref{sec: review of COAML}).

\subsection{Dynamic DARP and E-ADARP} \label{sec: review of dynamic VRP}

The dynamic DARP extends the static DARP by allowing service requests and system states to evolve over time. In this setting, decisions must be taken online, often under uncertainty, and solutions are repeatedly updated as new information becomes available. A central distinction in the literature concerns whether future service requests are anticipated or not. This distinction is independent of other sources of uncertainty, such as travel times, cancellations, or vehicle disruptions, as considered, for example, by \citet{Xiang2008} and \citet{Schilde2014}. We therefore distinguish between approaches that react only to currently available demand information and approaches that explicitly account for potential future requests.

\paragraph{Without anticipation of future service requests (non-anticipatory)}

Early work on the dynamic DARP without anticipation of future requests dates back to \cite{Psaraftis1980}, who developed a dynamic programming approach for a single-vehicle setting in which the solution is updated whenever a new request arrives, using only currently available information. A widely adopted strategy in this literature is to maintain a set of vehicle routes and update it as new requests become available, using fast insertion heuristics, repeated reoptimization, or combination of both, as illustrated in, among others, \cite{Coslovich2006}, \cite{beaudry2010}, \cite{Berbeglia2012}, \cite{Markovic2015}, and \cite{Souza2022}. More recently, \cite{bongiovanni2022machine} considered the electric extension of the dynamic DARP and proposed a large neighborhood search metaheuristic in which machine learning is used to guide the selection of neighborhood operators. \citet{huang2025genetic} proposed a genetic-programming hyper-heuristic for the dynamic electric DARP, learning vehicle- and request-allocation policies.

\paragraph{With anticipation of future service requests (anticipatory)}

A different line of research explicitly accounts for the anticipation of future demand, typically within a stochastic and dynamic framework. A general methodological foundation for anticipatory stochastic dynamic vehicle routing problems is provided by \cite{Ulmer2019,Ulmer2020}. Within the dynamic DARP and E-ADARP literature, anticipatory approaches differ notably in the timing at which decisions are made.

Many anticipatory approaches operate in an event-driven setting, in which the arrival of a new request triggers a new decision, possibly together with other operational events. \citet{Sayarshad2018} proposed a non-myopic dynamic DARP based on a Markov decision process and a multi-server queuing approximation, using an infinite-horizon estimate of opportunity costs and future profits to guide real-time vehicle assignment and routing decisions. \citet{ritzinger2022comparison} compared a sample-scenario approach and anticipatory waiting strategies for a dynamic patient transportation problem. Their results show that incorporating information about future requests improves solution quality, while simple waiting strategies can outperform the more sophisticated scenario-based approach. \citet{Heitmann2023} addressed the dynamic DARP with capacity management by combining value function approximation (VFA) and multiple-scenario approach (MSA): MSA guides routing decisions, while VFA governs the acceptance or rejection of incoming requests. Building on this, \citet{HEITMANN2024} proposed an improved VFA method that first reduces the dimensionality of the state representation and then progressively increases it, improving the computational performance and learning speed of the VFA component. \citet{Ackermann2025} proposed a multiple-plan approach that maintains several alternative feasible plans to preserve flexibility toward future requests. They also investigated an anticipatory MSA extension, but found no improvement over the original approach. \citet{schulz2026} considered a dynamic DARP that maximizes the number of accepted requests, where customer acceptance is modeled as a function of the relative detour compared to the direct travel time. They developed an efficient insertion method that anticipates future requests through a potential function measuring the attractiveness of current tours for possible future requests. Finally, \citet{Tomandl2026} studied a dynamic E-ADARP with penalties for late pickups, solving it via a rolling-horizon large neighborhood search guided by two reward functions learned through reinforcement learning, one incentivizing early charging and the other waiting, both taking as input the current state and the expected number of future requests.

A smaller number of works decouple decisions from individual request arrivals and instead process requests periodically. \citet{Tafreshian2021} proposed a proactive shuttle-dispatching framework with a discretized time horizon, where requests arriving during the preceding interval are processed jointly and vehicles may switch among routes constructed from forecast demand. Similarly, in the context of on-demand ride-pooling systems, \citet{Fielbaum2022} considered fixed assignment intervals and proposed anticipatory routing mechanisms that modify both vehicle--request assignments and vehicle routing so as to steer the fleet toward states that are better positioned for future demand.

These periodic approaches are closer to our decision timing, since several requests are processed jointly at fixed decision times. However, their decisions mainly concern assignment and routing. In our setting, dispatch timing is itself optimized: requests do not have to be served immediately after they arrive. At each decision epoch, the operator decides whether to serve them now or postpone them, so that new requests arriving later may create better ride-sharing opportunities.

\subsection{Combinatorial Optimization Augmented Machine Learning (COAML)}\label{sec: review of COAML}

In the context of combined combinatorial optimization (CO) and machine learning (ML) approaches, a well-known paradigm is the \textit{predict-then-optimize}, which predicts the unknown parameters and subsequently applies CO to a parametrized optimization model (e.g., \cite{bertsimas2020predictive}). In this sequential learning and optimization paradigm, the ML layer is typically trained to improve predictive accuracy without directly accounting for how prediction errors affect the subsequent optimization task. Consequently, predictions that are accurate from a statistical perspective are not necessarily those that lead to the best downstream decisions. To better align prediction and decision quality, \cite{elmachtoub2022smart} proposed \textit{smart predict-then-optimize} pipelines, which train the ML layer using a tailored loss function that quantifies the decision error. A detailed review of \textit{predict-then-optimize} and \textit{smart predict-then-optimize} can be found in \cite{sadana2025survey}. COAML provides a more general learning architecture than \textit{smart predict-then-optimize}. Both approaches rely on a CO oracle to map the output of the ML model to a decision. However, in \textit{smart predict-then-optimize}, the ML model is trained to predict the unknown parameters of the downstream optimization problem, typically its cost vector. In more general COAML architectures, the ML output does not need to correspond to observable optimization parameters and can instead represent learned scores or weights whose purpose is to guide the CO oracle toward high-quality feasible solutions. One challenge of applying COAML pipelines is to derive meaningful gradients of the respective CO layer to allow for backpropagation \citep{agrawal2019differentiable}. To handle this problem, recent approaches proposed stochastic perturbation techniques, including additive \citep{berthet2020learning} and multiplicative perturbations \citep{dalle2022learning}, which enable differentiation through CO layers. These approaches have been successfully applied to COAML pipelines in supervised learning settings with Fenchel-Young losses \citep{blondel2020learning}.

The COAML paradigm can be used to solve a wide variety of problem classes, ranging from hard (possibly deterministic) combinatorial problems to multi-stage stochastic optimization problems. Among the many applications of this paradigm, two works stand out as particularly close to the setting studied here. \citet{jungel2025learning} addressed online dispatching and rebalancing for autonomous mobility-on-demand fleets, modeling the decision problem as a $k$-disjoint shortest path problem with ML-predicted arc weights. \citet{baty2024combinatorial} tackled a dynamic vehicle routing problem with time windows (VRPTW) with dispatching waves, where the CO layer solves a prize-collecting VRPTW with learned request prizes.

\subsection{Conclusion and motivation}

To the best of our knowledge, existing works on the dynamic DARP and E-ADARP do not jointly capture the decision structure studied in this paper. In particular, existing approaches predominantly follow an event-driven setting, where the solution is reconsidered each time a new request arrives. In contrast, we consider a decision-epoch setting in which requests accumulated between consecutive epochs are processed jointly. This introduces an additional decision: at each epoch, the service provider must determine not only how to serve the currently available requests, but also which requests should be served now and which should be postponed, thereby shaping future ride-sharing opportunities.

The resulting decision problem also combines several operational features that, to the best of our knowledge, have not been jointly addressed in previous learning-based approaches: epoch-based request selection, pickup--drop-off coupling, ride-sharing, service-quality constraints, electric-vehicle charging, and the revision of ongoing vehicle routes. In particular, previously planned but unexecuted route segments are not treated as fixed decisions; they can be revised when new information becomes available, allowing newly selected requests to be inserted into routes of vehicles that have already been dispatched. Furthermore, the objective can account for both total vehicle travel time and excess user ride time, allowing operational efficiency to be balanced against passenger service quality.

These characteristics also distinguish our work from the COAML applications that are most closely related to ours. Neither \citet{jungel2025learning} nor \citet{baty2024combinatorial} accounts for ride-sharing among requests, whereas we explicitly model ride-sharing and its impact on service quality. Moreover, in \citet{baty2024combinatorial}, routes are fixed once vehicles are dispatched, and a dispatched vehicle cannot serve requests revealed in later epochs. In contrast, we retain the possibility of revising unexecuted route portions and inserting new requests into routes of vehicles that are already in operation.

\section{Dynamic E-ADARP: Problem Description and Formulation} \label{sec:DynEADARP}
In this section, we describe the Dyn-EADARP in Section \ref{subsec:problem_defn}.
Then, we formulate the Dyn-EADARP as a Markov Decision Process with finite horizon in Section \ref{subsec:math_formulation}. 

\subsection{Overall Description} \label{subsec:problem_defn}
The Dyn-EADARP is a dynamic variant of the static E-ADARP and shares several of its settings \citep{bongiovanni2019electric} which are summarized as follows:
\begin{itemize}
    \item It is defined on a directed graph $G=(V,A)$, where $V$ denotes the vertex set and $A = \{(i,j):i,j \in V, i \neq j\}$ denotes the arc set. Then, $V$ can be further partitioned into several subsets, i.e., $V=P\cup D\cup S\cup O\cup F$. Here, $P$ and $D$ denote the sets of pickup and drop-off nodes, respectively, $S$ is the set of recharging stations, and $O$ and $F$ are the sets of origin and destination depots. The set of customer nodes is denoted as $N$ and $N = P \cup D$.

    \item Each dial-a-ride request $i$ is represented by a pickup node $i\in P$ and its corresponding drop-off node $n+i\in D$, where $n$ is the total number of requests. The maximum ride time associated with request $i$ is denoted by $m_i$. Each node $i\in V$ is associated with a time window $[e_i,l_i]$, where $e_i$ and $l_i$ are the earliest and latest service start times, respectively, as well as a load $q_i$ and a service duration $s_i$. For each pickup node $i\in P$, $q_i>0$, while its corresponding drop-off node satisfies $q_{n+i}=-q_i$. For all nodes $j\in S\cup O\cup F$, we set $q_j=s_j=0$.

    \item Each vehicle $k\in K$ starts from an origin depot $o\in O$ and terminates at a destination depot $f\in F$. Vehicles may visit recharging stations in $S$ when needed and only when they have no passengers onboard. For each arc $(i,j)\in A$, let $t_{i,j}$ and $b_{i,j}$ denote the travel time and battery consumption, respectively. Battery discharge and recharge are assumed to be proportional to travel and charging times, respectively. Let $\alpha$ denote the recharging rate. 
    Partial recharging is allowed, and vehicles must reach their destination depot with a battery level of at least $\gamma Q$, where $\gamma$ is the minimum-battery-level ratio and $Q$ is the battery capacity. Recharging stations may be visited multiple times by vehicles. The triangle inequality is assumed to hold for both $t_{i,j}$ and $b_{i,j}$. Finally, the electric autonomous vehicles (EAVs) are homogeneous, with vehicle capacity $C$ and battery capacity $Q$.
\end{itemize}
Unlike the static E-ADARP, where information on all requests is known at the beginning of the planning horizon, the Dyn-EADARP considers an online setting in which requests are released over the planning horizon $[0,T]$. The planning horizon is divided into $E$ \textit{epochs} of equal length, denoted by
$\mathcal{E}={[\tau_0,\tau_1],[\tau_1,\tau_2],\ldots,[\tau_{E-1},\tau_E]}$, where $\tau_e$ denotes the start time of epoch $e$, for $e\in{0,\ldots,E-1}$. Incoming requests during $[\tau_{e-1},\tau_e]$ are collected and become available for decision making at time $\tau_e$. These requests are associated with a release time of $\tau_e$, indicating when their information becomes available to the system and when they become ready for service.

Let $R_e$ denote the set of requests available at time $\tau_e$, consisting of all requests that have been released but not yet selected for service. 
At the beginning of epoch $e$, the control center selects a subset $\mathcal{D}_e\subseteq R_e$ to be served, while the remaining requests $R_e\setminus\mathcal{D}_e$ are postponed to future epochs. It then determines the routing and scheduling decisions for the selected requests subject to the operational constraints (as defined in Section \ref{subsubsec:epoch_reoptimization}). We assume that the fleet is sufficiently large so that all requests can eventually be served within the planning horizon. To ensure non-anticipativity, a vehicle can travel from the predecessor node to the pickup node of a request only after the request has been released and selected for service. Meanwhile, in the Dyn-EADARP, unexecuted portions of previously planned routes may be revised at the beginning of the next epoch, as illustrated in Example~\ref{example1}. When entering a new epoch, each route is divided into a fixed portion and an adjustable portion. The fixed portion includes all nodes that have already been visited, as well as, for each vehicle, the upcoming node to be visited (e.g., node $D_1$ in Example \ref{example1}).
The remaining portion can be reoptimized using newly available information (e.g., the dashed path in Example \ref{example1}). This flexibility distinguishes our setting from many existing dynamic routing approaches (e.g., \cite{baty2024combinatorial}), in which previously planned routes are fixed once determined.


\begin{exmp}[Revise the adjustable portion of a vehicle route]\label{example1}
Consider one vehicle and three requests $r_i=(P_i,D_i)$, $i\in\{1,2,3\}$, with an epoch duration of 10 minutes. Request $r_3$ is released at decision epoch 1, which starts at $\tau_1=10$ and ends at 20, whereas $r_1$ and $r_2$ are released at epoch 0. The service duration at each node is ignored for simplicity. Figure~\ref{fig:example1} shows the time window of each node and the travel time of each arc. Let $B_{v}$ denote the service start time at node $v$ in the initial plan.

\begin{figure}[htbp]
\centering
\includegraphics[width=\textwidth]{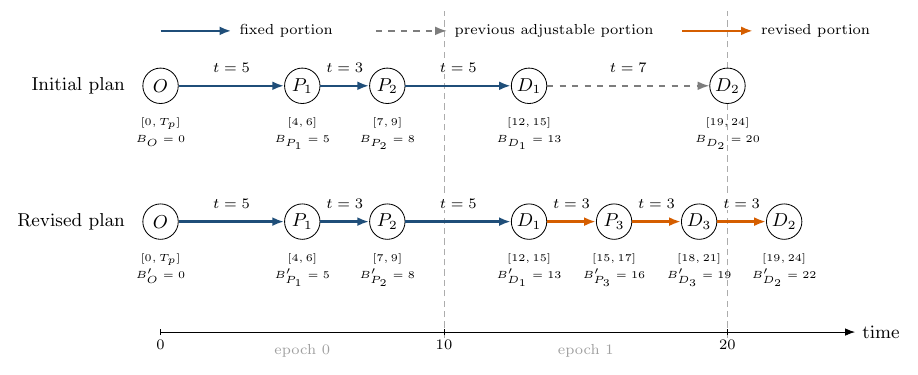}
\caption{\centering Illustration of route and schedule revision across decision epochs.}
\label{fig:example1}
\end{figure}
\end{exmp}

At epoch 0, requests $r_1$ and $r_2$ are selected for service, and the corresponding route and schedule are shown in the upper route. By the start of epoch 1, the last node that has a service start time no later than $\tau_1$ is $P_2$ and its successor node is $D_1$. Therefore, the route until $D_1$ is fixed while the remaining part (i.e., $D_2$) can be adjusted. At epoch 1, the newly released request $r_3$ is inserted immediately after $D_1$, and the remaining route and schedule (represented by $B_i'$) are revised accordingly. It should be noted that the fixed portion remains unchanged.

\subsection{Mathematical Formulation} \label{subsec:math_formulation}

\subsubsection{Markov decision process formulation} \label{subsubsec:markov_decision_process}

We formulate the Dyn-EADARP as a finite-horizon Markov decision process, where the system evolves over a sequence of decision epochs and the decision at each epoch is made based on the current state of the system. To facilitate understanding, we introduce several terms:
\paragraph{\textbf{System state}}
At the beginning of epoch $e$, the system state is defined as
\[
\state=(\tau_e,R_e,\Xi_e),
\]
where $\Xi_e$ is the current fleet state including the current location, load, battery level, and the current plan (i.e., committed nodes and planned schedule) for each EAV.

\paragraph{\textbf{Feasible action}}
Given the system state $\state$ at $\tau_e$, the service provider determines a feasible action $\boldsymbol{a}_e$, including a feasible decision $\boldsymbol{y}_e$ and an associated feasible plan $\Pi_e$ for vehicle routing, scheduling, and charging, given $\boldsymbol{y}_e$:
\[
\boldsymbol{a}_e = (\boldsymbol{y}_e, \Pi_e)
\]
where the feasible decision $\boldsymbol{y}_e$ is a vector that indicates which request is selected at epoch $e$ to serve
\[
\boldsymbol{y}_e = (y_i)_{i \in R_e},
\qquad
y_i =
\begin{cases}
1, & \text{if request } i \text{ is selected},\\
0, & \text{otherwise},
\end{cases}
\quad \forall i \in R_e.
\]
Therefore, deciding $\boldsymbol{y}_e$ is equivalent to choosing a set of requests $\mathcal{D}_e \subseteq R_e$ to serve. Set $\mathcal{D}_e$ includes two types of requests:
\begin{itemize}
    \item Must-dispatch requests which are mandatory to be selected as they cannot be postponed (Definition \ref{must_dispatch_request}).
    \item Postponable requests which can be postponed to the next epoch and are optional to be selected.
\end{itemize}
\begin{defn}[\Must requests] \label{must_dispatch_request}
Let $T_{i,e+1}'$ denote the earliest possible service start time at the pickup node of request $i$ if it is postponed to epoch $(e+1)$,
which is defined as
\[
T_{i,e+1}'
=
\max (e_i, \tau_{e+1}
+ \min_{o\in O}t_{o,i})
\]
where $t_{o,i}$ is the travel time from origin depot $o$ to pickup node $i$ so that $\tau_{e+1}
+ \min_{o\in O}t_{o,i}$ gives the earliest possible arrival time with a newly assigned vehicle.
Then, request $i$ is called \Must if
\[
T_{i,e+1}' > l_i
\qquad \text{or} \qquad
T_{i,e+1}' + s_i + t_{i,n+i} > l_{n+i}.
\]
\end{defn}
The first condition captures time-window infeasibility on pickup node $i$ after postponement, while the second indicates the time-window infeasibility on drop-off node $n+i$, even with the earliest possible pickup time at $i$ and direct travel $t_{i,n+i}$.
After selecting $\mathcal{D}_e$, the decision maker has to determine a feasible plan $\Pi_e$ which requires solving an epoch-level optimization problem, as described in Section~\ref{subsubsec:epoch_reoptimization}.

\paragraph{\textbf{System evolution}}
The plan $\Pi_e$ is executed during $[\tau_e,\tau_{e+1})$. At the beginning of epoch $(e+1)$, the fixed portion of each route is determined and remains unchanged, while the adjustable part can be revised based on newly released requests. The vehicle locations, loads, battery levels, onboard requests, and request set are then updated accordingly, yielding the next state $\boldsymbol x_{e+1} = (\tau_{e+1},R_{e+1},\Xi_{e+1})$.

\paragraph{\textbf{Policy and objective}}
Let $\mathcal{X}$ denote the state space. A policy
$\pi:\mathcal{X}\rightarrow\mathcal{Y}$ maps each state $\state$ to a
feasible decision $\pi(\state)\in\mathcal{Y}(\state)$. Our objective is to find a policy $\pi$ that minimizes the expected total cost over the horizon $[0, T]$.
\begin{equation}
\min_{\pi}\;
\mathbb{E}_{\pi}
\left[
\sum_{e=0}^{E-1}c_e(\pi(\state))
\right]
\label{dyn_mdp_obj}
\end{equation}
where $c_e$ denotes the total
cost incurred by all vehicles during $[\tau_e,\tau_{e+1})$ with policy $\pi$. In our paper, we consider a weighted-sum objective including total travel time and total excess user ride time, as in \cite{bongiovanni2019electric}. The excess user ride time is defined as the difference between actual user ride time and direct travel time $t_{i,n+i}$.

\subsubsection{Epoch-level optimization problem} \label{subsubsec:epoch_reoptimization}
At the beginning of epoch $e$, the service provider optimizes the adjustable portions of the vehicle routes based on the current system state $\state$ and the decision $\boldsymbol y_e$. To determine which part of a route is adjustable when entering epoch $e$, we use the planned service start times of its nodes. Specifically, let $a_k^e$ denote the last node (excluding the depot) on the current route whose service start time is no later than $\tau_e$. We further denote the immediate successor of $a_k^e$ on the current route by $a_k^{e,+}$. Then, all nodes up to and including $a_k^{e,+}$ define the fixed portion of the route (denoted by $\Gamma_k^{e,\mathrm{fix}}$), while the remaining portion can be adjusted based on newly revealed information (denoted by $\Gamma_k^{e,\mathrm{adj}}$).
Specifically, for each node on $\Gamma_k^{e,\mathrm{fix}}$, its time window is fixed at the actual service start time. 
The reoptimization of vehicle $k$'s route is therefore anchored at $a_k^{e,+}$, whose service start time and associated load and battery state are inherited from the previous plan. All requests following $a_k^{e,+}$ may be revised.
For an empty route, the fixed portion is empty and the vehicle remains available at its depot at time $\tau_e$.

The epoch-level problem follows the static E-ADARP formulation of \citet{bongiovanni2019electric}, including standard flow-conservation, time-window, pickup--drop-off precedence, maximum-ride-time, vehicle-capacity, battery, charging, and minimum-battery-level constraints.
Based on the static E-ADARP formulation, several adaptations are made to accommodate the dynamic setting. 
A request is included in the planned routes if and only if it is selected for service. A vehicle may travel from the predecessor node to a pickup node only after the corresponding request is released and selected.
Table \ref{tab:notation} summarizes our notations and the corresponding description.

\begin{table}[htbp]
\centering
\begin{threeparttable}
\caption{\centering Summary of main notations}
\label{tab:notation}
\renewcommand{\arraystretch}{0.85}
\small
\begin{tabular}{@{} p{0.28\textwidth} p{0.7\textwidth} @{}}
\toprule
\textbf{Notation} & \textbf{Description} \\
\midrule
$G=(V,A)$ & Directed graph. \\
$P,D,N,S,O,F$ & Pickup, drop-off, customer nodes, charging station, origin, and destination depot sets. \\
$R_e,\mathcal{D}_e$ & Available and selected request sets at epoch $e$. \\
$T,E,\tau_e$ & Planning horizon, number of epochs, and start time of epoch $e$. \\
$[e_i,l_i],s_i,m_i$ & Time window, service duration, and maximum ride time. \\
$K,C,Q$ & Vehicle set, vehicle capacity, and battery capacity. \\
$q_i$ & Load change at node $i$. \\
$t_{i,j},b_{i,j}$ & Travel time and battery consumption on arc $(i,j)$. \\
$\alpha$ & Recharging rate. \\
$\gamma Q$ & Minimum battery level at destination depot. \\
$a_k^e,a_k^{e,+}$ & Last fixed node and its successor on vehicle $k$'s route. \\
$\Gamma_k^{e,\mathrm{fix}},
 \Gamma_k^{e,\mathrm{adj}}$ & Fixed and adjustable route portions of vehicle $k$. \\
$\state=(\tau_e,R_e,\Xi_e)$ & System state at epoch $e$. \\
$\boldsymbol{a}_e = (\boldsymbol{y}_e,\Pi_e)$ & Action at epoch $e$. \\
$\boldsymbol{y}_e$ & Decision at epoch $e$. \\
$\Pi_e$ & Routing, scheduling, and charging plan associated with $\boldsymbol{y}_e$. \\
$\pi$ & Decision policy. \\
$c_e$ & Stage cost at epoch $e$. \\

\bottomrule
\end{tabular}
\end{threeparttable}
\end{table}

\section{Policy Architecture}
\label{sec:policy_architecture}
The Dyn-EADARP is challenging not only because of its combinatorial complexity, but also because decisions must be made under uncertainty about future requests. To efficiently make these decisions, we develop a policy based on the COAML framework. The COAML architecture combines an upstream statistical model with a downstream CO layer, allowing the statistical model to learn state-dependent information that directly guides the resulting operational decisions.
The overall framework is illustrated in Figure~\ref{pipelines}. At each epoch $e$, the system state $\state$ is passed to a statistical model $\varphi_w$, which predicts a score vector $\boldsymbol{\theta}_e$. The predicted scores are then used to parameterize the CO layer, which returns a feasible decision $\boldsymbol{y}_e$. During training, the statistical model is optimized through a loss function that measures the discrepancy between the resulting decision $\boldsymbol{y}_e$ and its target $\bar{\boldsymbol{y}}_e$. 
\begin{figure}[ht]
\centering
\includegraphics[width=\textwidth]{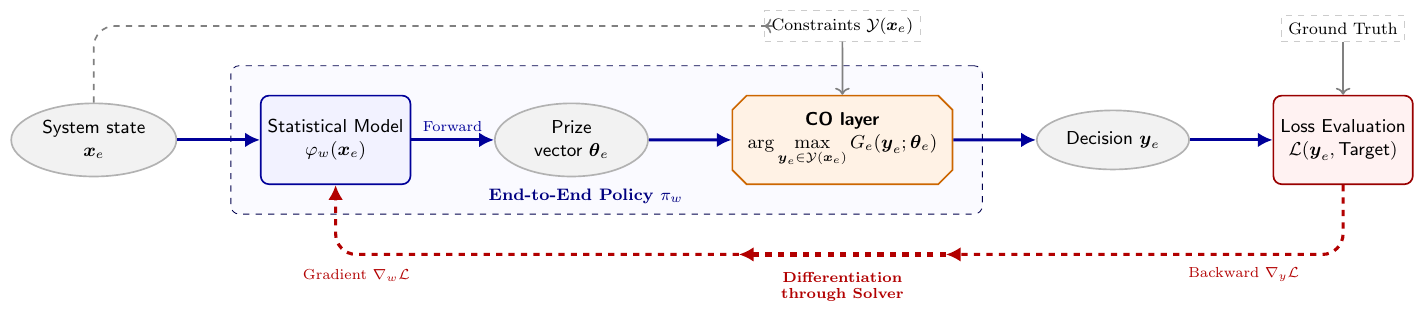}
\caption{\small General COAML architecture as presented in \cite{schiffer2026combinatorial}. A statistical model parameterizes a surrogate optimization problem. We propagate gradients backwards through the CO layer, enabling end-to-end training of the policy $\pi_w$ to minimize a downstream decision loss}
\label{pipelines}
\end{figure}

In our setting, the key decision involves which requests should be selected for service (i.e., $\boldsymbol{y}_e$). This request-selection decision involves an inter-epoch trade-off: serving a request at the current epoch incurs an immediate routing cost, but can also affect future ride-sharing and routing opportunities.
To capture this trade-off without explicitly solving the full dynamic problem, we associate each request $i\in R_e$ with a prize $\theta_{i,e}$, which represents the learned value of selecting this request at the current epoch.

Given these prizes, the CO layer jointly selects requests and constructs the corresponding routing, scheduling, and charging plan. 
Since each request may either be selected or postponed, with selected requests contributing a prize while inducing additional routing cost, the resulting epoch-level problem naturally takes the form of a prize-collecting E-ADARP.
The statistical model $\varphi_w$ predicts the prize vector $\boldsymbol{\theta}_e$ from the current state $\state$, while the CO layer converts these prizes into a feasible decision. Proposition~\ref{proposition1} further shows that this prize-based representation is enough and any optimal epoch-level decision of the Dyn-EADARP can be represented as an optimal solution to a prize-collecting E-ADARP for an appropriate prize vector.

In the following, we first present the prize-collecting E-ADARP and the corresponding CO layer in Section~\ref{subsec:CO_layer}. We then introduce the statistical model used to predict request prizes in Section~\ref{subsec:statsistical_model}. 

\subsection{Prize-collecting E-ADARP objective and CO layer}\label{subsec:CO_layer}

Given the predicted prize $\theta_{i,e}$ for each released and unserved request $i\in R_e$, the prize-collecting E-ADARP determines which requests to serve and constructs the corresponding routing, scheduling, and charging plan. Its objective is to maximize the total net profit which balances the total prize collected from selected requests against the resulting travel cost:
\begin{equation}
\label{prize-collection objective}
\max_{\boldsymbol y_e \in\mathcal{Y}(\state)} G_e(\boldsymbol{y}_e;\boldsymbol{\theta}_e) = \boldsymbol{\theta}_e^\top \boldsymbol{y}_e -
c_e(\boldsymbol{y}_e).
\end{equation}
Given the system state $\state$ and the predicted prize vector $\boldsymbol{\theta}_e$, the CO layer applies a heuristic algorithm $\mathcal{H}$ (Section \ref{subsec:heuristic_oracle}) to the prize-collecting E-ADARP and returns a feasible decision $\boldsymbol{y}_e\in\mathcal{Y}(\state)$:

\begin{equation}
\label{co_layer}
\boldsymbol{y}_e
=
\mathcal{H}\bigl(\state,\boldsymbol{\theta}_e\bigr),
\qquad
\boldsymbol{y}_e\in\mathcal{Y}(\state).
\end{equation}
Similar to \cite{baty2024combinatorial}, we also have the following proposition for the Dyn-EADARP:
\begin{proposition} \label{proposition1}
    For any $\state$, there exists a vector $\boldsymbol{\theta} \in \mathbb{R}^{|{R}_e|}$ such that any optimal solution of \eqref{prize-collection objective} is an optimal decision with respect to the original problem \eqref{dyn_mdp_obj}.
\end{proposition}
\begin{proof}
Assume the optimal solution of the original problem \eqref{dyn_mdp_obj} is $\boldsymbol{y}_e^* = (y^*_1, \ldots, y^*_n)$, then by assigning $\boldsymbol{\theta}_e = (\theta_1, \ldots, \theta_n)$ where 
$\theta_i = M$ if $y^*_i = 1$ and $-M$ otherwise for all $i=1,\ldots, n$. Then, it is clear that when $M$ is big enough, the optimal solution of \eqref{prize-collection objective} is $\boldsymbol{y}_e^*$.
\end{proof}

\subsection{Statistical models} \label{subsec:statsistical_model}
The choice of statistical model $\varphi_w$ is generic and determines how contextual information and problem features are encoded into the score vector $\boldsymbol{\theta}$ that interacts with the optimization oracle. A simple example of $\varphi_w$ is a generalized linear model, which parameterizes the output for each request based on its features (see Section \ref{subsec:feature_engineering}). Let $\feature(i,\state)$ denote the feature vector associated with request $i$ and the current system state. The resulting architecture can be written as follows and can be trained directly with any CO layer:
\[
\varphi_w(\state)
=
\left(\boldsymbol w^\top \feature(i,\state)\right)_{i\in R_e},
\]
where $\boldsymbol w$ denotes the learnable weight vector.

More flexible prediction models can also be incorporated into the proposed architecture. In particular, we consider a simple multilayer perceptron (MLP) and a graph neural network (GNN) as alternative statistical models. The MLP applies a small feed-forward network to the features of each request and can capture nonlinear relationships between request characteristics and their predicted prizes. In contrast, the GNN allows information to be exchanged among requests through graph-based message passing, and can therefore capture interactions and dependencies among requests, such as their spatial or temporal compatibility. Implementation details are provided in Section~\ref{subsec:extended_analysis}.

\section{Learning Algorithm} \label{sec:learning_algorithm}
In this section, we first introduce the general learning approach in Section \ref{subsec:general_approach}. Then, we present the loss function used to train the policy in Section \ref{subsec:loss_function}. Then, we describe the proposed heuristic oracle to solve the prize-collecting E-ADARP in Section \ref{subsec:heuristic_oracle}. Finally, we present the generation of anticipative training samples in Section \ref{subsec:anticipative_target}.

\subsection{General approach} \label{subsec:general_approach}
We now describe the general learning approach used to train the ML–CO policy. The statistical model $\varphi_w$ maps a system state $\state$ to a prize vector $\boldsymbol{\theta}_e$, which is then passed to the CO layer to generate a feasible decision $\boldsymbol{y}_e$. 
The objective is to learn parameters $w$ such that the prize vectors predicted by $\varphi_w$ induce request-selection decisions that closely match high-quality target decisions.
In the supervised learning setting, we are given $N$ training pairs
\[
\left\{
\left(\boldsymbol x^{(d)},\bar{\boldsymbol y}^{(d)}\right)
\right\}_{d=1}^{N},
\]
where $\boldsymbol x^{(d)}$ is a sampled system state and $\bar{\boldsymbol y}^{(d)}$ is the
corresponding target decision. The model parameters $w$ are learned by
minimizing
\begin{equation}
\label{eq:sl_score_loss}
\min_w
\frac{1}{N}
\sum_{d=1}^{N}
\mathcal{L} \left(
\varphi_w(\boldsymbol x^{(d)}),
\bar{\boldsymbol y}^{(d)}
\right).
\end{equation}
where $\mathcal{L}$ is a loss function.

A key challenge is that the decision returned by the discrete CO layer is piecewise constant with respect to the predicted prizes, which makes direct gradient-based training difficult \citep{schiffer2026combinatorial}. We therefore use a perturbed surrogate loss that preserves the link between the predicted prizes and the resulting decisions while providing useful gradient information \citep{dalle2022learning}. The loss and its gradient are presented in the next part.

\subsection{Loss function} \label{subsec:loss_function}

For a state $\state$ with predicted prize vector
$\boldsymbol{\theta}_e=\varphi_w(\state)$, 
a natural loss is the gap between the objective value obtained by the CO layer under the prize vector $\boldsymbol{\theta}_e$ and the objective value of the fixed target decision $\bar{\boldsymbol y}_e$ evaluated under the same prizes:
\begin{equation}
\label{eq:unperturbed_loss}
\mathcal{L}
(\boldsymbol{\theta}_e,\bar{y}_e)
=
\max_{\boldsymbol{y}_e\in\mathcal{Y}(\state)}
G_e(\boldsymbol{y}_e;\boldsymbol{\theta}_e)
-
G_e(\bar{\boldsymbol y}_e;\boldsymbol{\theta}_e).
\end{equation}
However, Equation \eqref{eq:unperturbed_loss} involves optimization over a discrete feasible set $\mathcal{Y}(\state)$, making the loss nonsmooth with respect to $\boldsymbol{\theta}_e$. 
Indeed, $\arg\max_{y \in \mathcal{Y}(\state)}G_e(\boldsymbol{y}_e ; \boldsymbol{\theta}_e)$ is piecewise constant in $\boldsymbol{\theta}_e$ and the gradient is almost zero everywhere. Consequently, directly differentiating through
the argmax does not provide a suitable learning signal.
Following \cite{baty2024combinatorial} and \cite{jungel2025learning}, we perturb the predicted prizes by $\epsilon Z$ to obtain a smoother learning signal, where $Z$ is a standard Gaussian vector of the same dimension
as $\boldsymbol{\theta}_e$ and $\epsilon>0$ controls the perturbation
magnitude. The perturbed loss is defined as
\begin{equation}
\label{eq:perturbed_loss}
\mathcal{L}_{\epsilon}
(\boldsymbol{\theta}_e,\bar{y}_e)
=
\mathbb{E}_{Z}
\left[
\max_{\boldsymbol{y}_e\in\mathcal{Y}(\state)}
G_e
\left(
\boldsymbol{y}_e;
\boldsymbol{\theta}_e+\epsilon Z
\right)
\right]
-
G_e
\left(
\bar{y}_e;
\boldsymbol{\theta}_e
\right).
\end{equation}

Following the same idea as in \cite{baty2024combinatorial}, we can show that $\mathcal{L}_{\epsilon}$ is convex in $\boldsymbol{\theta}_e$ and derive its gradient. Proposition \ref{prop:perturbed_gradient} shows that this gradient has a simple interpretation.

\begin{proposition}[Gradient of the perturbed loss]
\label{prop:perturbed_gradient}
Suppose that the perturbed CO problem is solved exactly. For a realization
of $Z$, define
\begin{equation}
\label{eq:perturbed_optimal_decision}
\boldsymbol{y}_e^*(Z)
\in
\arg\max_{\boldsymbol{y}_e\in\mathcal{Y}(\state)}
G_e
\left(
\boldsymbol{y}_e;
\boldsymbol{\theta}_e+\epsilon Z
\right).
\end{equation}
The gradient of
$\mathcal{L}_{\epsilon}$ with respect to the predicted prize vector
$\boldsymbol{\theta}_e$ is
\begin{equation}
\label{eq:perturbed_gradient}
\nabla_{\boldsymbol{\theta}_e}
\mathcal{L}_{\epsilon}
=
\mathbb{E}_{Z}
\left[
\boldsymbol{y}_e^*(Z)
\right]
-
\bar{\boldsymbol{y}}_e.
\end{equation}
\begin{proof}
    See \ref{app:FY_loss}.
\end{proof}
\end{proposition}

Proposition~\ref{prop:perturbed_gradient} shows that the learning signal is the difference between the expected request-selection vector induced by the perturbed prizes and the target request-selection vector. In practice, for computational efficiency, our CO layer uses a heuristic (Section~\ref{subsec:heuristic_oracle}) instead of an exact solver. For each state, we draw $N_{\mathrm{mc}}$ independent perturbations and compute
\begin{equation}
\label{eq:heuristic_perturbed_decision}
\widetilde{\boldsymbol{y}}_e^{(q)}
=
\mathcal{H}
\left(
\state,
\boldsymbol{\theta}_e+\epsilon Z^{(q)}
\right),
\qquad
q=1,\ldots,N_{\mathrm{mc}}.
\end{equation}
The gradient used for training is approximated by
\begin{equation}
\label{eq:surrogate_gradient}
\widetilde{\nabla}_{\boldsymbol{\theta}_e}
\mathcal{L}_{\epsilon}
=
\frac{1}{N_{\mathrm{mc}}}
\sum_{q=1}^{N_{\mathrm{mc}}}
\widetilde{\boldsymbol{y}}_e^{(q)}
-
\bar{\boldsymbol{y}}_e,
\end{equation}
which serves as a surrogate gradient for updating the statistical model.

\subsection{Heuristic oracle for the prize-collecting E-ADARP} \label{subsec:heuristic_oracle}
The heuristic oracle must satisfy two main requirements. First, it must be computationally efficient, since for each training sample it is called multiple times to solve perturbed prize-collecting E-ADARP instances when evaluating the loss in Equation~\eqref{eq:perturbed_loss}. Second, the fixed portion $\Gamma_k^{e,\mathrm{fix}}$ of existing routes must remain unchanged, while only the adjustable portions can be revised subject to the operational constraints described in Section~\ref{subsubsec:epoch_reoptimization}. We therefore develop an efficient greedy heuristic $\mathcal{H}$ that incorporates an exact feasibility-checking procedure, as detailed in Section \ref{subsubsec:feasibility_check}.

\subsubsection{Algorithm design}
The heuristic $\mathcal{H}$ takes as input the current system state $\state$ and predicted prize vector $\boldsymbol{\theta}_e$ for the currently released and unserved requests.
It returns a decision $\widetilde{\boldsymbol{y}}_e$ along with updated vehicle routes and service start times. 
The heuristic consists of several steps, summarized as follows:

\paragraph{\textbf{Step 1: Identify the fixed and adjustable portions for each dispatched vehicle}} At the beginning of epoch $e$, we first identify $\Gamma_k^{e,\mathrm{fix}}$ and  $\Gamma_k^{e,\mathrm{adj}}$ of each dispatched vehicle $k \in K$, as described in Section \ref{subsubsec:epoch_reoptimization}.

\paragraph{\textbf{Step 2: Insert \Must requests}} Must-dispatch requests are processed in increasing order of latest pickup time windows. For each \Must request, we enumerate all pickup and drop-off insertion positions in $\Gamma_k^{e,\mathrm{adj}}$ for dispatched vehicles, as well as the option of assigning the request to a new vehicle. 
    
Each time a request is inserted into a route, we perform the feasibility check described in Section~\ref{subsubsec:feasibility_check}. If the resulting route violates only the battery constraints, we attempt to repair feasibility by inserting a recharging station following the procedure detailed in Section~\ref{subsubsec:insertion_recharging_station}. Note that at least one feasible insertion always exists, since we assume a sufficiently large fleet size and the request can be assigned to an empty route served by a newly dispatched vehicle.

Among all feasible insertions, we select the one with the minimum marginal cost $\Delta C_i$, defined as the increase in cost caused by inserting the request and a recharging station (if necessary):
\[
\Delta C_i = \mathrm{Cost}(r^{\mathrm{new}})-\mathrm{Cost}(r^{\mathrm{orig}}),
\]
where $r^{\mathrm{orig}}$ denotes the original route and $r^{\mathrm{new}}$ denotes the resulting route after the insertion. The routes are then updated before processing the next request.

\paragraph{\textbf{Step 3: Select and insert \Postpone requests based on net profit}} After inserting all \Must requests, we sort the \Postpone requests in $R_e$ in descending order of their predicted prizes. Because the adjustable portions of the routes can be revised at subsequent epochs, it is unnecessary to construct, at the current epoch, a complete plan for requests that would only be picked up in later epochs. 
We therefore restrict attention to \Postpone requests that can be picked up within the current epoch. For each such request $i$, we perform the same insertion procedure as above and accept the insertion only if it yields a positive net profit, i.e., $\theta_i > \Delta C_i$.

\begin{exmp}
We illustrate the insertion process using the example in Figure~\ref{fig:request_insertion}. For simplicity, we consider one dispatched vehicle and one request released at epoch 1 (i.e., request 3). We first identify $\Gamma_k^{e,\mathrm{fix}}$ and $\Gamma_k^{e,\mathrm{adj}}$, represented by the blue and dashed lines, respectively. Upon entering epoch 1, only $\Gamma_k^{e,\mathrm{adj}}$ can be revised. Suppose request 3 is a \Postpone request with predicted prize $\theta_3$. We enumerate all feasible insertion positions in $\Gamma_k^{e,\mathrm{adj}}$ and select the one with the minimum marginal cost $\Delta C_3$ (which also accounts for any additional cost incurred by a recharging visit), shown in orange. Since $\theta_3>\Delta C_3$, request 3 is inserted, resulting in the lower route.
\begin{figure}[ht]
\centering
\includegraphics[width=\textwidth]{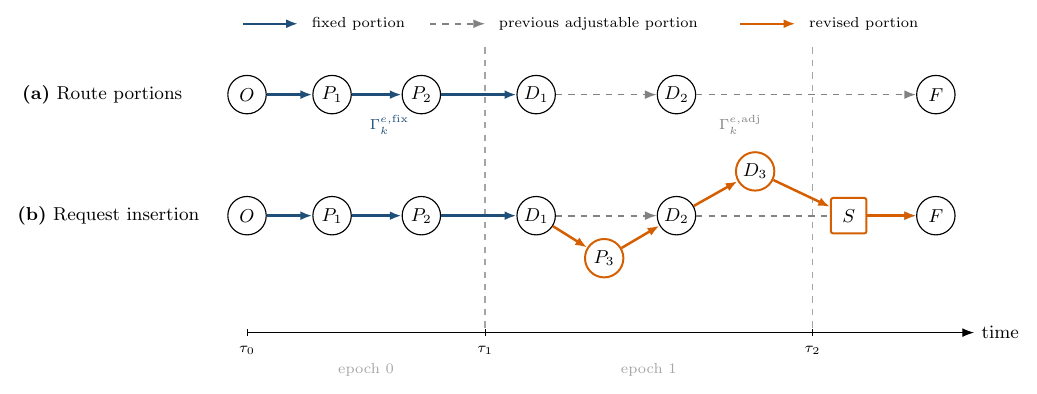}
\caption{\centering Illustrative example of inserting a request}
\label{fig:request_insertion}
\end{figure}

\end{exmp}

It is worth noting that more advanced procedures, such as local search, may improve routing quality within each epoch. However, such improvements do not necessarily lead to better overall performance across epochs and can substantially increase computational time, especially for instances with many requests and vehicles. Therefore, we adopt the greedy insertion procedure described above, which provides a good balance between solution quality and computational efficiency, as demonstrated in the evaluation (Section \ref{subsec:performance_analysis}).


\subsubsection{Feasibility check}
\label{subsubsec:feasibility_check}

We adapt the fragment-based path representation of \citet{su2023deterministic} to perform feasibility checks at each request insertion. Their approach is based on the notion of a \BRFrag, as defined in Definition~\ref{fragment}, and consists of two main steps. First, all feasible \BRFrags are precomputed. Then, since each route can be decomposed into a sequence of such fragments, as illustrated in Figure~\ref{battery-restricted fragment example}, the route-feasibility check for the E-ADARP is transformed into an E-VRP feasibility problem by abstracting the fragments as arcs in an auxiliary graph. The resulting problem is solved using Proposition 1 of \cite{su2023deterministic}.

In addition to the feasibility check, \cite{su2024branch} developed a method to calculate an earliest feasible service schedule that minimizes excess user ride time (Theorem 1 of \cite{su2024branch}). We use this schedule as the planned service schedule of the route, which is then used to determine its fixed and adjustable portions in the next decision epoch.

\begin{defn}[Battery-restricted fragment, \cite{su2023deterministic}] \label{fragment}
Assume that $\frag = (i_1,i_2, \cdots,i_k)$ is a sequence of pickup and drop-off nodes, where the vehicle arrives empty at $i_1$ and leaves empty at $i_k$ and has passenger(s) on board at other nodes. Then, we call ``$\frag$'' a \BRFrag if there exists a feasible route of the form:
$$(o,s_{i_1},\cdots,s_{i_v},\overbrace{i_1,i_2, \cdots,i_k}^{\frag},s_{i_{v+1}},\cdots,s_{i_m},f),$$ where $s_{i_1},\cdots,s_{i_v},s_{i_{v+1}},\cdots,s_{i_m} (v,m \geqslant 0)$ are recharging stations, and $o \in O$, $f \in F$. An illustrative example is shown in Figure \ref{battery-restricted fragment example}, which contains two \BRFrags, i.e., $\frag_1 = \{P_1,P_2,D_1,D_2\}$ and $\frag_2 = \{P_3,D_3\}$.
\end{defn}
\begin{figure}[H]
\centering
\includegraphics[width=14cm]{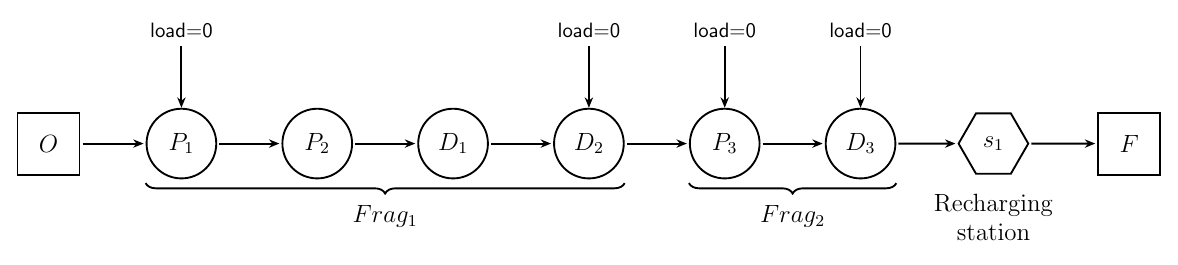}
\caption{\centering Example of \BRFrags}
\label{battery-restricted fragment example}
\end{figure}

To apply the method of \citet{su2023deterministic} to our dynamic setting, we make the following adaptations:

\begin{itemize}

    \item \textbf{Handling fixed route portions:}
    Unlike the static setting, where all requests are known at the beginning of the planning horizon, the dynamic setting may contain route portions that have already been executed or are being executed at the beginning of an epoch. To account for this, we make the following modifications: for each request $r \in \Gamma_k^{e,\mathrm{fix}}$, its service start time is fixed to the value specified in the previous plan. 
    Specifically, if the planned service time of request $r$ is $t_0$, its time window is replaced by $[t_0,t_0]$ before applying the feasibility-checking procedure of \citet{su2023deterministic}.
    This modification ensures that the newly generated route remains identical to the previous plan over its fixed portion.

    \item \textbf{On-demand fragment checking:}
    In previous work \citep{su2023deterministic}, all feasible fragments are precomputed and then directly queried during route-feasibility checking. However, the number of possible fragments grows exponentially with the number of requests. In our dynamic setting, where new requests are released at each epoch, repeating this precomputation can become computationally expensive for large instances. We therefore adopt an on-demand fragment-checking strategy. Rather than enumerating all possible fragments in advance, we only evaluate the fragments that appear in the route currently being checked, using the feasibility rules of \citet{su2023deterministic}. To further reduce repeated computations, we maintain an online cache that stores all previously checked fragments, including both feasible and infeasible ones.

    This strategy substantially reduces the fragment-generation effort in practice, as shown in Table~\ref{tab:a3-36-m27-acceleration}. On an instance with 140 requests, the total evaluation time of the trained policy is reduced to 1.47 seconds, which is approximately $6400\times$ faster than the original implementation that performs full fragment precomputation at each epoch, while producing the same objective value.

\begin{table}[htbp]
    \centering
    \caption{\centering Comparison of the accelerated and original feasibility-checking procedures on an instance with 140 requests}
    \label{tab:a3-36-m27-acceleration}
    \begin{tabular}{lcc}
        \toprule
        Metric & Accelerated version & Original version \\
        \midrule
        Total time (s)           & \textbf{1.47} & 9409.77 \\
        Speedup                  & \textbf{$\approx 6400\times$} & -- \\
        Requests served          & 140/140 & 140/140 \\
        Objective                & 1896.6 & 1896.6 \\
        Feasible fragments stored & 518 & 456,071 \\
        \bottomrule
    \end{tabular}
\end{table}

\end{itemize}

\subsubsection{Recharging station insertion} \label{subsubsec:insertion_recharging_station}
For a candidate route with a newly inserted request, if only battery infeasibility is detected, we try to repair the route by inserting a recharging station. If any other infeasibility is detected, the candidate route is immediately discarded without applying this repair procedure. The insertion of a recharging station is performed as follows:
\begin{itemize}
    \item \textbf{Identify possible positions}: Since a recharging station can only be visited when the vehicle has no passenger, we decompose the candidate route into fragments. The nodes in $\Gamma_k^{e,\mathrm{fix}}$ as well as their service start times remain fixed and cannot be modified. Therefore, a recharging station can only be inserted within $\Gamma_k^{e,\mathrm{adj}}$ and between two consecutive fragments or before the destination depot.
    \item \textbf{Insert a recharging station}: We examine all possible insertion positions in a backward order, starting from the position immediately before the destination depot and moving toward the first possible position. At each position, we select the recharging station that yields the minimum detour and check the feasibility of the insertion following Proposition~1 of \cite{su2023deterministic}. If the resulting route is feasible, the procedure terminates.
\end{itemize}

\subsection{Anticipative solution as training target} \label{subsec:anticipative_target}
The Dyn-EADARP is hard to solve as we only have limited information until the current epoch. However, if we know all requests at the beginning of the horizon, the Dyn-EADARP reduces to a static E-ADARP with release time constraints (hereafter anticipative Dyn-EADARP). By solving the anticipative Dyn-EADARP, we can obtain an anticipative policy based on perfect information. The detailed formulation of the anticipative Dyn-EADARP is shown in \ref{anticipative Dyn-EADARP}. To generate the training dataset, we solve each training instance (introduced in Section \ref{subsec:dynamic_instance_generation}) by iterating the Deterministic Annealing (DA) algorithm of \cite{su2023deterministic} 10,000 times, where we consider the release time constraints in the time window tightening process (as shown in Section \ref{subsec:dynamic_instance_generation}). The obtained anticipative solutions are decomposed into epoch-level state--decision pairs. Specifically, at each decision epoch, we extract the corresponding system state and the request-selection decisions prescribed by the anticipative solution for the currently available requests, with each state--decision pair forming one training sample.


\section{Numerical Results}\label{sec:numerical_results}
This section presents the numerical experiments for evaluating the proposed ML--CO policy. We first describe the generation of the dynamic instances (Section \ref{subsec:dynamic_instance_generation}), the benchmark policies (Section \ref{subsec:benchmarks}), and the features (Section \ref{subsec:feature_engineering}) used by the learning model. We then evaluate all considered policies on the testing instances and compare their performance with an approximate lower bound (Section \ref{subsec:performance_analysis}). Next, we conduct extended experiments to examine the effect of the ML model architecture and the size of the training dataset (Section \ref{subsec:extended_analysis}). Finally, we perform sensitivity analyses (Section \ref{subsec:sensiticvity_analysis}) on several key modeling choices and parameters, including decision frequency, minimum-battery-level ratio $\gamma$, and the weight vector $(w_1,w_2)$ assigned to the two objective components.

For the main experiments, we use a linear ML model, with an epoch duration of 60 minutes and a minimum battery ratio of $\gamma=0.1$. We later vary the epoch duration and $\gamma$ in Sections~\ref{subsubsec:different_epoch_length} and~\ref{subsubsec:different_battery_levels}, respectively, to investigate the effects of decision frequency and the minimum battery requirement on the results. In the weighted-sum objective, the weights of total travel time and total excess user ride time are set to 1 and 0, respectively, such that the main experiments focus on minimizing total travel time. We further consider, in Section \ref{subsubsec:effect_of_weights}, weight combinations of $(0.75,0.25)$ and $(0.5,0.5)$ to examine the trade-off between vehicle travel time and excess user ride time.
The model is trained using the Adam optimizer with a learning rate of $1\times10^{-3}$. We set the perturbation parameter to $\varepsilon=0.5$, use five perturbation samples per instance, and train for at most 30 epochs with an early-stopping patience of 8. All experiments are conducted on a workstation running Ubuntu 24.04.1 with an Intel Core i9-14900K CPU and 16 GiB of memory.

\subsection{Dynamic instance generation} \label{subsec:dynamic_instance_generation}
We construct dynamic instances from the static type-a and type-r E-ADARP instances in \cite{su2023deterministic}, where the largest instance contains 8 vehicles and 96 requests. For each static instance, we divide the planning horizon $[0,T]$ into equal-duration epochs of duration $\eta$. We then make $J$ attempts to generate new requests in each epoch and keep only feasible requests. All generated requests are assigned a release time equal to the beginning of the epoch. Each request is generated independently as follows:
\begin{itemize}
\item \textbf{Sample locations and time windows:} The pickup and drop-off locations are independently sampled from the pickup and drop-off locations of the original static instance. A pair of pickup and drop-off time windows is then randomly sampled from an existing request in the same static instance.

\item \textbf{Adjust time windows:} The sampled time windows are adjusted to match the new pickup and drop-off locations. The pickup time window is shifted according to the change in travel time from the depot to the pickup location. Similarly, the drop-off time window is shifted according to the change in travel time from the pickup to the drop-off location. This adjustment preserves the time-window pattern of the original request while accounting for the new locations.

\item \textbf{Check request feasibility:} Each generated request is checked against its release time, and infeasible requests are discarded. For requests generated in the last epoch, we additionally check that they can be completed within the planning horizon.
\item \textbf{Time window tightening:} After generating the requests, their pickup and drop-off time windows are further
tightened according to the release-time, maximum ride-time, precedence, and planning-horizon constraints.

\end{itemize}

The detailed generation procedure is provided in \ref{app:dynamic_instances}.

For each static instance $I_{\mathrm{static}}$, we consider two sample-size levels $J=\lceil 0.5n\rceil$ and $J=\lceil 0.75n\rceil$, where $n$ is the number of requests in $I_{\mathrm{static}}$. For each pair $(I_{\mathrm{static}},J)$, we generate 10 dynamic instances using different random seeds. Thus, each static instance gives rise to 20 dynamic instances in total. With 24 static instances, this procedure generates 480 dynamic instances. Among these instances, 160 small- to medium-sized dynamic instances (built upon 8 static instances) are used to generate training samples (in total 1280 samples), while the remaining instances are used for testing. In the instance generation, we set epoch duration $\eta$ to 60 minutes and the largest dynamic instance contains 469 requests. To facilitate organization, we refer to each pair $(I_{\mathrm{static}},J)$ as an instance family in the following experiments. For each instance family, we report the average performance of a considered policy.

\subsection{Benchmark policies} \label{subsec:benchmarks}
We evaluate the performance of our proposed policy against a set of benchmarks summarized as follows. All baseline policies are also allowed to modify the adjustable portion of each dispatched vehicle.
\paragraph{Greedy policy}
At each epoch, the greedy policy selects all newly-released requests as \Must requests. The insertion process of selected requests remains the same as in Section \ref{subsec:heuristic_oracle}. 

\paragraph{Lazy policy}
In contrast with the greedy policy, the lazy policy only selects \Must requests which cannot be delayed anymore at each epoch. The same insertion process is applied as in Section \ref{subsec:heuristic_oracle}.

\paragraph{Random policy}
At each epoch, the random policy randomly selects 50\% \Postpone requests as well as all \Must requests of this epoch to serve. The same insertion process is applied as in Section \ref{subsec:heuristic_oracle}.

\paragraph{Rolling horizon policy}
At epoch~$e$, the rolling-horizon policy first samples requests from epoch~$(e+1)$ until the last epoch following the process described in Section~\ref{subsec:dynamic_instance_generation}. Together with the requests in $R_e$, these sampled requests form an anticipative scenario covering epoch~$e$ to the end of the horizon, which is then solved using the DA algorithm with 500 iterations. The resulting anticipative routes are used only to determine which requests should be selected at epoch $e$. Specifically, all \Must requests are selected, while a \Postpone request is selected if and only if, in the anticipative solution, it is assigned to the same route as a \Must request of epoch $e$.

\subsection{Features} \label{subsec:feature_engineering}
The features include request-level features and global information about the epoch. We consider 25 features which can be computed or obtained directly from current state $\state$. Considered features for request $i$ are summarized in Table \ref{tab:features}.
\begin{table}[!ht]
\centering
\adjustbox{max width=\textwidth}{\begin{threeparttable}
\caption{\centering Summary of features for a request $i$}
\scriptsize
\label{tab:features}
\begin{tabular}{lll}
\toprule
Index & Feature & Description \\
\hline
1--2 & Pickup coordinates & Pickup $x-$ and $y-$coordinate \\
3--4 & Drop-off coordinates & Drop-off $x-$ and $y-$coordinate \\
5 & Demand & $q_i$ \\
6 & Pickup service time & $s_i$ \\
7 & Drop-off service time & $s_{n+i}$ \\
8--9 & Pickup time window & $e_i,l_i$ \\
10--11 & Drop-off time window & $e_{n+i},l_{n+i}$ \\
12 & Depot-to-pickup time & $\min_{o\in O}t_{o,i}$ \\
13 & Pickup-to-dropoff time & $t_{i,n+i}$ \\
14 & Dropoff-to-depot time & $\min_{f\in F}t_{n+i,f}$ \\
15 & Pickup time window width & $l_{i} - e_{i}$ \\
16 & Time window tightness ratio & $t_{o,i} /l_{i}$ \\
17 & Battery consumption & $\min_{o\in O}b_{o,i} + b_{i,n+i} + \min_{f\in F}b_{n+i,f}$ \\
18 & Nearest charging station from pickup & $\min_{s \in S} t_{i,s}$ \\
19 & Nearest charging station from drop-off & $\min_{s \in S} t_{n+i,s}$\\
20 & Ride time slack & $m_i - t_{i,n+i}$ \\
21 & Remaining time ratio & $(l_{i} - \tau_e) / T$ \\
22 & Must-dispatch indicator & see Definition \ref{must_dispatch_request}\\
23 & Epoch progress & $e /E$ \\
24 & Epoch start time & $\tau_e$ \\
25 & Epoch end time & $\tau_{e+1}$ \\
\bottomrule
\end{tabular}
\end{threeparttable}}
\end{table}

\subsection{Performance Analysis}\label{subsec:performance_analysis}
We report the aggregated results of the different policies in Table \ref{tab:aggregated_results} and Figures \ref{fig:relative_to_anticipative} and \ref{fig:average_objective}. In Table \ref{tab:aggregated_results}, the column named \textit{Avg. obj} reports the average objective value obtained by each method, while \textit{Avg. evaluation time} gives the average computational time required to solve one dynamic instance. The last column \textit{ML--CO imp} reports the cost saving achieved by ML--CO relative to the other baselines. 
For better illustration and to assess the potential optimality gap, we compute an \lb by solving the anticipative problem with the DA algorithm for 5,000 iterations, using the ML--CO solution as a warm start. For each instance, we perform three independent DA runs and use their average objective value as the anticipative reference value (hereafter Ref).
The mean gap is computed as the arithmetic mean of the relative gaps on 320 testing instances to the Ref and is reported in the third column (\textit{Avg. gap to Ref}).

From Table \ref{tab:aggregated_results} and Figure \ref{fig:relative_to_anticipative}, ML--CO achieves the best performance among all benchmark policies, with an average objective value of 2378.6 and an average gap of only 4.8\% to the Ref. Among the four benchmarks introduced in Section~\ref{subsec:benchmarks}, the lazy policy performs best which suggests that delaying dispatching decisions is actually an effective policy for the Dyn-EADARP. However, ML--CO further reduces the cost by 8.7\% relative to the lazy policy. It indicates that ML--CO captures useful decision patterns (e.g., when a request should be dispatched earlier) by learning from anticipative targets. This finding is consistent with \cite{baty2024combinatorial} and extends their observation to the more complex Dyn-EADARP setting.

Figure \ref{fig:average_objective} further shows the average performance of different policies on each instance family. For all the instance families, the ML--CO policy outperforms all benchmark policies. For detailed results on each instance family, readers can refer to \ref{app:detailed_results}.

\begin{table}[!ht]
\centering
\begin{threeparttable}
\caption{\centering Aggregated performance over all 320 testing instances.}
\label{tab:aggregated_results}
\begin{tabular}{lcccc}
\toprule
Method & Avg. obj & Avg. evaluation time (s) & Avg. gap to Ref (\%) & ML--CO imp (\%) \\
\midrule
ML--CO & \textbf{2378.6} & 6.3 & \textbf{+4.8} & --- \\
Rolling horizon & 2741.5 & 529.6 & +20.1 & +13.2 \\
Greedy & 2776.0 & 15.0 & +21.6 & +14.3 \\
Random & 2737.0 & 11.2 & +20.0 & +13.1 \\
Lazy & 2605.5 & 2.6 & +14.6 & +8.7 \\
\bottomrule
\end{tabular}
\end{threeparttable}
\end{table}

\begin{figure}[!ht]
    \centering
    \subfigure[Aggregated performance to Ref]{
    \includegraphics[width=13cm]{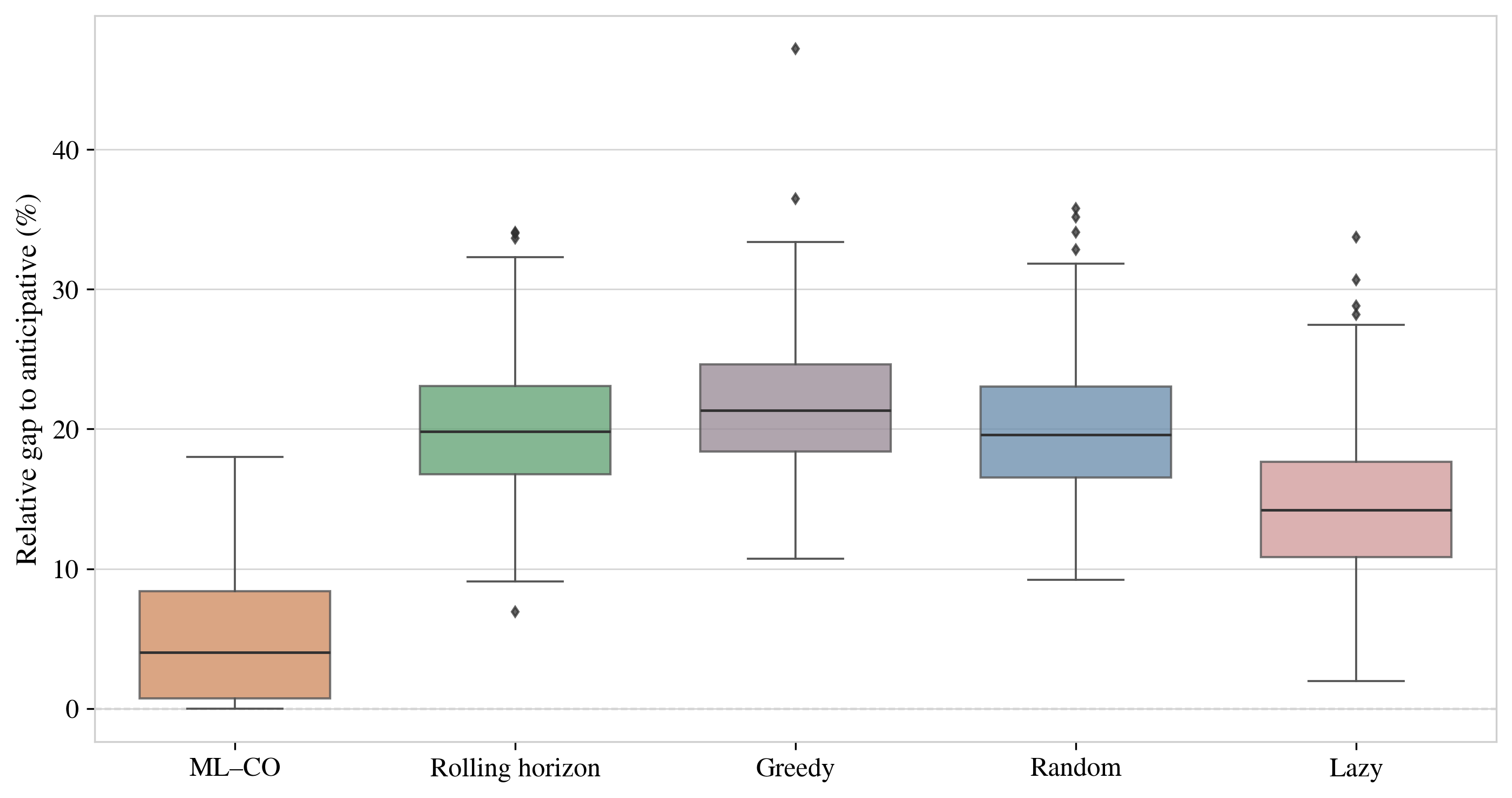}
    \label{fig:relative_to_anticipative}
    }
    \vfill
    \subfigure[Average objective gap to Ref for each instance family]{
    \includegraphics[width=\textwidth]{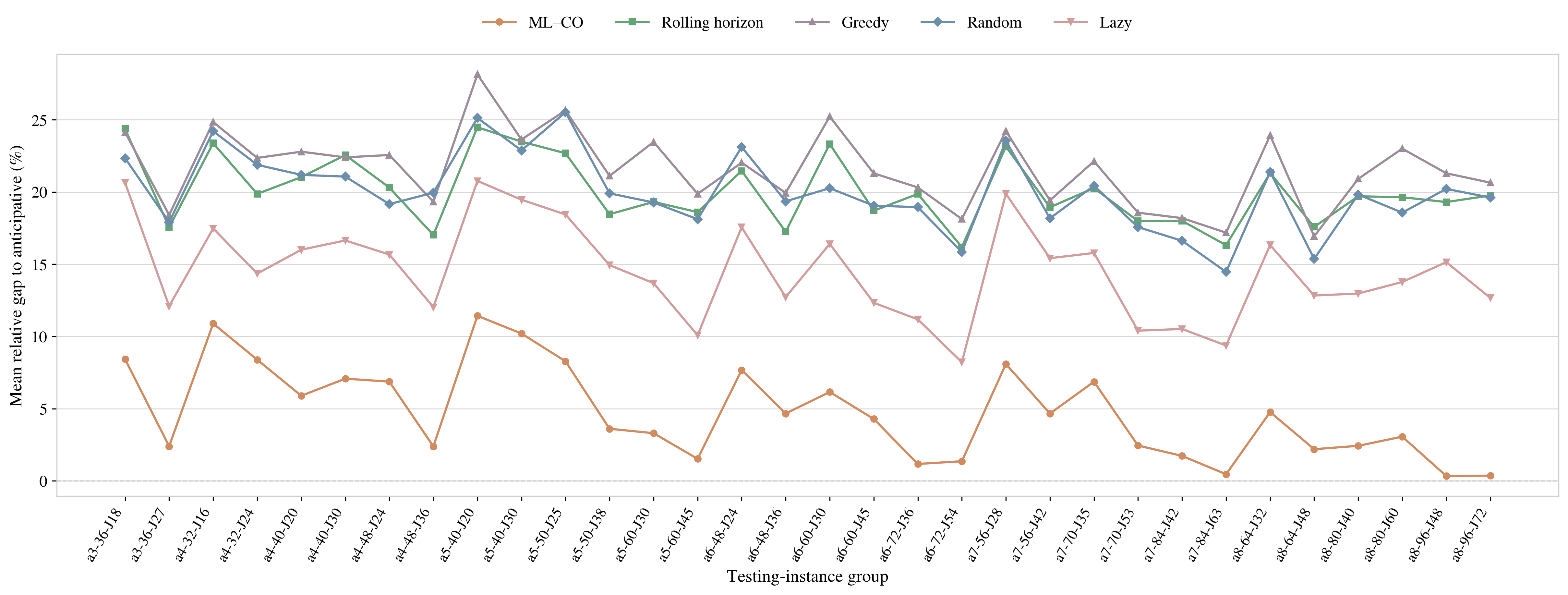}
    \label{fig:average_objective}
    }
    \caption{\centering The relative performance among different policies to Ref}
\end{figure}

\subsubsection{Case study}
We use a case study to better understand two aspects of the ML--CO policy: (1) how it achieves its performance advantage over the benchmark policies, and (2) why imitation of an anticipative policy works well in practice. 

\paragraph{\textbf{Timing of selection}}
 The results in Table \ref{tab:request-selection-timing} show clear differences in request-selection timing across policies. Except for the lazy policy, all other baseline policies tend to select requests relatively early. The greedy policy is the extreme case, with all requests selected immediately upon release. 

This differs from \cite{baty2024combinatorial}, where rolling horizon performs well while the lazy policy performs poorly. The main reason is the problem structure. In our Dyn-EADARP, postponing a request can be beneficial because it may be grouped with future requests, creating better ride-sharing opportunities and reducing routing costs. In contrast, serving requests earlier is more valuable in the dynamic VRP studied by \cite{baty2024combinatorial}, as it reduces the risk of time-window violations. Under the rolling-horizon policy, an available request may be selected early simply because it shares an anticipative route with a \Must request in the current epoch. Once selected, part of its planned route may become fixed in the next epoch, limiting future adjustments and reducing ride-sharing opportunities, leading to increased routing costs.

In contrast to the greedy policy, the lazy policy works better than other baselines as it postpones requests as much as possible until they must be served. Our ML--CO policy also postpones most requests but serves 10.5\% of requests earlier than the lazy policy. This selection strategy leads to a 8.7\% reduction in the mean objective, suggesting that ML--CO can identify requests for which earlier service creates better routing opportunities.


\begin{table}[!ht]
\centering
\begin{threeparttable}
\caption{\centering Comparison of request selection timing across policies}
\label{tab:request-selection-timing}
\renewcommand{\arraystretch}{1.10}
\small
\begin{tabular}{@{}lccccc@{}}
\toprule
\textbf{Metric}
& \textbf{Greedy}
& \textbf{Random}
& \textbf{Rolling horizon}
& \textbf{Lazy}
& \textbf{ML--CO} \\
\midrule

Avg. wait epochs
& 0.0
& 0.7
& 0.3
& 2.8
& 2.7 \\

Selected upon release
& 100.0\%
& 59.5\%
& 71.5\%
& 19.1\%
& 21.7\% \\

Selected in final epoch
& 3.9\%
& 7.7\%
& 5.4\%
& 33.1\%
& 30.8\% \\

Selected earlier than Lazy
& 80.9\%
& 65.6\%
& 73.8\%
& --
& 10.5\% \\

\bottomrule
\end{tabular}

\begin{tablenotes}
\footnotesize
\item ``Avg. wait epochs'' denotes the average number of epochs
between request release and selection. The last row reports the percentage
of requests selected earlier than under Lazy.
\end{tablenotes}
\end{threeparttable}
\end{table}

\paragraph{\textbf{Detailed example on a4-32-J16-rep07}}
We analyze the medium-sized testing instance a4-32-J16-rep07. For this instance, our ML--CO policy achieves an objective value of 851.98, whereas the Lazy policy obtains 914.68, corresponding to a 7.36\% higher cost than ML--CO.
For comparison, we also report the anticipative solution. The objective value of the obtained anticipative solution is 815.129. From the obtained solutions, all considered policies finally use the same number of vehicles, therefore isolating the effect of routing and scheduling.

Comparing with the solution obtained by the lazy policy, we find that ML--CO policy changes the selection timing of only seven requests, as summarized in Table \ref{tab:advanced-requests-mlco}.  
We also summarize the epoch at which the anticipative solution selects these requests. Interestingly, six out of seven requests are selected in exactly the same epochs in the anticipative and ML-CO solutions. This suggests that ML--CO learns some of the key selection patterns of the anticipative solution. Moreover, ML--CO achieves a 7.36\% cost reduction compared with the lazy policy. For example, at epoch 4, both ML--CO and the anticipative policy select requests 39 and 50 and assign them to a dispatched vehicle, enabling a compact ride-sharing route. Under the lazy policy, both requests are postponed until epoch 5 and are inserted into two different routes, resulting in higher travel time. This example suggests that ML--CO is able to identify requests for which postponement has a high opportunity cost. By serving such requests just one epoch earlier, ML--CO creates better ride-sharing opportunities and improves overall routing efficiency. It therefore suggests an empirical insight that the ML--CO policy actually learns which requests should no longer be postponed (even though they are postponable).

\begin{table}[htbp]
\centering
\begin{threeparttable}
\caption{\centering Comparison of selected epochs for requests advanced by ML--CO for the instance a4-32-J16-rep07}
\label{tab:advanced-requests-mlco}
\renewcommand{\arraystretch}{1.05}
\small
\begin{tabular}{@{}lccc@{}}
\toprule
\textbf{Requests}
& \textbf{Anticipative selection}
& \textbf{ML--CO selection}
& \textbf{Lazy selection} \\
\midrule
$\{6,12,13\}$ & Epoch 0 & Epoch 0 & Epoch 1 \\
$\{30\}$      & Epoch 2 & Epoch 2 & Epoch 3 \\
$\{41\}$      & Epoch 4 & Epoch 3 & Epoch 4 \\
$\{39,50\}$   & Epoch 4 & Epoch 4 & Epoch 5 \\
\bottomrule
\end{tabular}
\end{threeparttable}
\end{table}

\subsection{Extended Analysis} \label{subsec:extended_analysis}
In this section, we analyze how different training settings affect the performance of the ML--CO policy. In Section~\ref{subsubsec:different_statistical_models}, we investigate the effect of using more expressive statistical models in the ML layer, including MLP and GNN models. In Section~\ref{subsubsec:different_training_instances}, we vary the size of the training set to evaluate the generalization ability of the learned policy.

\subsubsection{Different statistical models} \label{subsubsec:different_statistical_models}
We also evaluate the performance of ML--CO policy integrating other statistical models (i.e., MLP and GNN). The input of all considered statistical models is the feature matrix $\psi(\state)$ corresponding to requests in $R_e$, with request $i \in R_e$ associated with a feature vector $\psi(i,\state)$. In \textit{linear model} $\varphi_w = (w^\top\psi(i,\state))_i$, each request-level feature vector $\psi(i,\state)$ is converted to its prize $\theta_i$ by the same linear mapping. In \textit{MLP}, each request-level feature vector $\psi(i,\state)$ is passed through two fully connected layers of 10 hidden units with ReLU activations, followed by a scalar output layer, yielding the request prize $\theta_i$. All requests share the same network parameters, while their prizes are computed independently.
In GNN, we construct an epoch-specific directed graph in which each node represents a request in the current request set $R_e$. The graph therefore changes dynamically across epochs. 
We construct a sparse graph by connecting only requests that can potentially be served consecutively.
Specifically, a directed edge $(i,j)$ for two requests $i$ and $j$ is included only if the partial path $(i,n+i,j)$ is feasible. 
For each feasible edge $(i,j)$, we assign the weight
\[
a_{ij}=\exp\left(-\frac{t_{n+i,j}}{\tau}\right),
\]
where $t_{n+i,j}$ is the travel time from the drop-off of request $i$ to the pickup of request $j$, and $\tau$ is the average such travel time among requests in the current epoch. 
We then apply two message-passing layers, allowing each request to aggregate information from its neighboring compatible requests. Finally, the resulting node representation is passed through an MLP to obtain the request prize. Unlike the linear model and MLP, which predict each request prize independently, the GNN produces prizes that incorporate information from neighboring compatible requests.

As shown in Table \ref{tab:expanded-model-architecture-aggregated}, the GNN-based ML–CO policy achieves the best overall performance, reducing the mean gap to the Ref from 4.8\% with the linear model to 4.6\%.
However, the differences among the three statistical models remain small. One reason might be that the structural information among requests is considered in the decision process of the CO oracle (Section \ref{subsec:heuristic_oracle}). The GNN additionally incorporates such interactions when predicting the prizes, leading to further but small improvements.

\begin{table}[!ht]
\centering
\begin{threeparttable}
\caption{\centering Aggregated performance of the ML--CO statistical models over all 320 testing instances.}
\label{tab:expanded-model-architecture-aggregated}
\begin{tabular}{ccccc}
\toprule
Statistical model & Avg. objective & Avg. gap to Ref (\%) & Training time (s) & Avg. evaluation time (s) \\
\midrule
Linear & 2378.6 & +4.8 & 2092.3 & 6.3 \\
MLP & 2378.6 & +4.9 & 4084.8 & 7.2 \\
Sparse GNN & \textbf{2374.1} & \textbf{+4.6} & 3505.3$^{a}$ & 7.4 \\
\bottomrule
\end{tabular}
\begin{tablenotes}
\scriptsize
\item a: The GNN has a shorter training time than the MLP because training terminates earlier with early-stop patience parameters.   
\end{tablenotes}
\end{threeparttable}
\end{table}

\subsubsection{Different number of training instances}
\label{subsubsec:different_training_instances}
We also analyze the impact of the number of training instances on the generalization performance of the ML--CO policy. We consider four training set sizes: 40, 80, 120, and 160 instances. The full 160-instance training set is generated from eight static instances. The 40-, 80-, and 120-instance training sets are constructed incrementally from the full training set using stratified sampling \citep{cochran1977sampling}, ensuring that all eight static instances and both $J$-levels are included at each training-set size. Starting from 40 instances, we successively add 40 instances following the same sampling scheme to obtain the 80- and 120-instance training sets, while the 160-instance set contains all available training instances.

From Table \ref{tab:training-size-aggregated}, we observe that the performance of the ML–CO policy is stable across the considered training-set sizes.
Increasing the size of the training set from 40 to 160 instances only slightly improves the mean objective and the gap to the Ref. This suggests that the model already learns useful decision patterns from a relatively small training set.

\begin{table}[!ht]
\centering
\begin{threeparttable}
\caption{\centering Effect of the number of training instances on the linear ML--CO model over all 320 testing instances.}
\label{tab:training-size-aggregated}
\begin{tabular}{cccc}
\toprule
Training-set sizes & Avg. objective & Avg. gap to Ref (\%) & Training time (s) \\
\midrule
40 & 2380.5 & +4.9 & 660.1\\
80 & 2380.7 & +4.9 & 893.2\\
120 & 2378.9 & +4.8 & 1338.3\\
160 & \textbf{2378.6} & +4.8 & 2092.3\\
\bottomrule
\end{tabular}
\end{threeparttable}
\end{table}


\subsection{Sensitivity studies} \label{subsec:sensiticvity_analysis}
To further examine the impact of key modeling and operational settings, we conduct a series of sensitivity studies. Specifically, we investigate varying the routing flexibility (i.e., allowing revisions to previously planned routes or not) in Section \ref{subsubsec:effec_of_adjusting}, varying the decision frequency (i.e., different epoch durations) in Section \ref{subsubsec:different_epoch_length}, varying minimum-battery-level restrictions (i.e., different $\gamma$ values) in Section  \ref{subsubsec:different_battery_levels}, and varying weight vectors in Section \ref{subsubsec:effect_of_weights}.
These analyses provide additional insights into how different operational settings affect the performance of the proposed approach.

\subsubsection{Effect of different levels of routing flexibility} \label{subsubsec:effec_of_adjusting}
To evaluate the effect of routing flexibility, we compare three ML--CO variants. In the \textit{adjustable} setting, the adjustable portion $\Gamma_k^{e,\mathrm{adj}}$ can be revised as new information becomes available, and newly selected requests can be incorporated into routes of already dispatched vehicles. In the \textit{semi-lock} setting, we no longer distinguish fixed or adjustable portions and we fix all previously-made routing decisions (including their ordering and schedules). Newly selected requests are allowed to be inserted into existing routes, but only after all previously planned nodes.
Finally, in the \textit{full-lock} setting, previously planned routes remain fixed, and newly selected requests must be served by newly dispatched vehicles. Ride-sharing is still allowed among requests assigned to the same newly dispatched vehicle. 
The last variant is consistent with the setting considered by \cite{baty2024combinatorial}. 
Table \ref{tab:non-revisable-mlco-aggregated} reports the average objective value, number of dispatched vehicles, and relative gap to the Ref over the 320 testing instances.

\begin{table}[!ht]
\centering
\begin{threeparttable}
\caption{Aggregated performance of ML--CO under different levels of routing flexibility.}
\label{tab:non-revisable-mlco-aggregated}
\begin{tabular}{cccc}
\toprule
Method & Avg. objective & Avg. dispatched vehicles & Avg. gap to Ref (\%) \\
\midrule
Adjustable ML--CO & \textbf{2378.6} & 18.7 & \textbf{+4.8} \\
Semi-lock ML--CO & 2620.9 & 17.8 & +14.9 \\
Full-lock ML--CO & 3097.4 & 65.3 & +36.4 \\
\bottomrule
\end{tabular}
\end{threeparttable}
\end{table}

Allowing greater routing flexibility substantially improves performance. The mean gap to the Ref decreases from 36.4\% under the full-lock setting to 14.9\% under the semi-lock setting, and further to 4.8\% under the adjustable setting. The large fleet size in the full-lock setting (65.3 vehicles on average) mainly results from serving selected requests at each epoch by dispatching new vehicles.
Allowing these requests to be inserted into existing routes, as in the semi-lock setting, reduces the average number of dispatched vehicles to 17.8.

More importantly, allowing the adjustable portions of existing routes to be revised further reduces the mean objective from 2620.9 to 2378.6. This improvement is achieved despite a similar number of dispatched vehicles (17.8 versus 18.7), indicating that the benefit mainly comes from better coordination of planned but unexecuted routing decisions rather than reduced vehicle numbers. Figure \ref{fig:revisible_v.s_non_revisible} shows that this performance ordering is consistent across all testing-instance groups. These results highlight both the value of reusing dispatched vehicles and the additional benefit of revising unexecuted routing decisions.

\begin{figure}[!ht]
\centering
\includegraphics[width=0.9\textwidth]{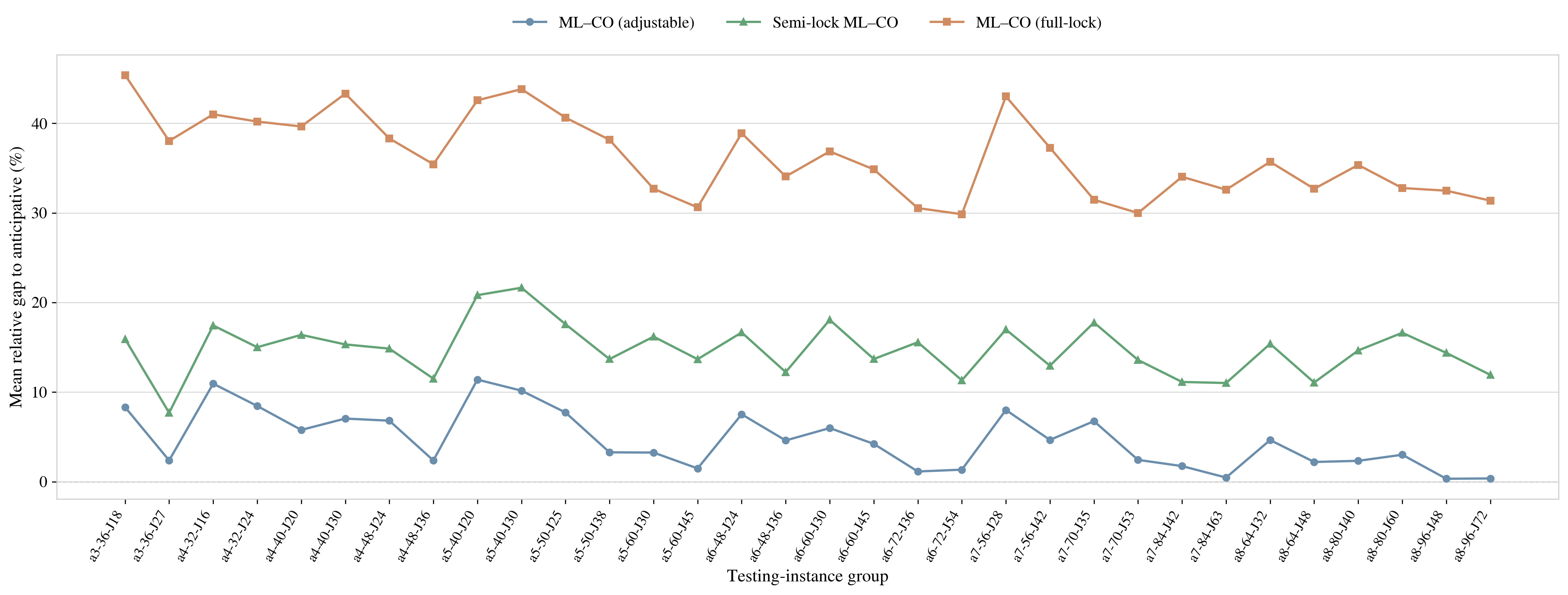}
\caption{\centering Relative gap to Ref under different levels of routing flexibility.}
\label{fig:revisible_v.s_non_revisible}
\end{figure}

\subsubsection{Effect of different epoch durations}
\label{subsubsec:different_epoch_length}
Since the default epoch duration is set to 60 minutes, decisions are updated only once per hour. We therefore further investigate the impact of epoch duration on the solution quality of the Dyn-EADARP. In addition to the default setting $\eta=60$, we consider three shorter epoch durations, $\eta \in \{10,20,30\}$, to evaluate the effect of more frequent decision-making. To isolate this effect, we maintain the original request release times and only introduce additional decision epochs to shorten the epoch duration. For example, when $\eta=30$, the original epoch $[0,60)$ is divided into $[0,30)$ and $[30,60)$, while all requests originally released at time 0 remain released at time 0. 

For each shorter epoch duration, we train a separate model and evaluate it on testing instances constructed under the corresponding fixed-release setting. For each value of $\eta$, we use 40 training instances to generate the training samples and evaluate the trained model on 32 testing instances. The training instances are selected following the same procedure as in Section \ref{subsubsec:different_training_instances}. The testing instances are constructed from 32 original 60-minute-epoch instances, corresponding to the first replication of each instance family $(I_{\mathrm{static}},J)$.
Detailed procedures for instance generation, training, and evaluation are provided in \ref{app:epoch_instance_generation}. We summarize our results in Table \ref{tab:epoch_length_results}.



\begin{table}[htbp]
    \centering
    \begin{threeparttable}
        \caption{\centering Effect of epoch duration on the ML--CO policy.}
        \label{tab:epoch_length_results}
        \small
        \setlength{\tabcolsep}{5pt}
        \begin{tabular}{cccccc}
            \toprule
            \multirow{2}{*}{\shortstack{epoch duration (min)}}
            & \multicolumn{3}{c}{Objective}
            & \multicolumn{2}{c}{Statistics of ML--CO} \\
            \cmidrule(lr){2-4}
            \cmidrule(lr){5-6}
            &\multirow{1}{*}{ ML--CO obj.}
            & \multirow{1}{*}{Lazy obj.}
            & \multirow{1}{*}{\shortstack{Gain over Lazy (\%)}}
            & \multirow{1}{*}{\shortstack{Avg. vehicles}}
            & Avg. cost per vehicle\\
            \midrule
            10 & 2626.0 & 2839.6 &  6.8 & 27.3 &  96.3 \\
            20 & 2507.1 & 2776.7 &  9.1 & 23.9 & 104.7 \\
            30 & 2414.1 & 2734.4 & 11.4 & 21.3 & 113.4 \\
            60 & 2459.7 & 2681.3 &  8.0 & 19.0 & 129.7 \\
            \bottomrule
        \end{tabular}
    \end{threeparttable}
\end{table}

From Table \ref{tab:epoch_length_results}, we observe that the ML--CO policy consistently outperforms the best baseline policy, Lazy, with gains ranging from 6.8\% to 11.4\%. These results also reveal an important practical implication regarding the choice of epoch duration: more frequent decision making does not necessarily lead to better performance. In general, shorter epoch durations result in higher costs. One possible reason is that, with more frequent decision epochs, the policy tends to dispatch more vehicles, as shown in ``Avg. vehicles'' column, which increases the total cost. However, the 30-minute setting is an exception to this general trend. Compared with the 60-minute setting, it uses 12.1\% more vehicles on average (21.3 versus 19.0), while reducing the average cost per vehicle by 12.6\% (113.4 versus 129.7). As a result, the 30-minute setting achieves a lower mean objective (2414.1 versus 2459.7) despite using more vehicles. This suggests that neither a shorter nor a longer epoch duration is necessarily better. Instead, an appropriate epoch duration should balance the number of dispatched vehicles with the routing efficiency of each vehicle.

\subsubsection{Effect of different minimum-battery-level restrictions}
\label{subsubsec:different_battery_levels}
We further investigate the effect of the minimum battery-level requirement. In addition to the baseline setting of $\gamma=0.1$, we consider two higher minimum battery ratios at the destination depot, $\gamma=0.4$ and $0.7$, following \cite{bongiovanni2019electric}. A higher $\gamma$ requires vehicles to have more battery upon returning to the destination depot, which may lead to more frequent visits to recharging stations.

For each value of $\gamma$, we use the same training instance set and generate training samples under the corresponding battery requirement. A separate ML--CO model is then trained for each $\gamma$ setting. For testing, we use the same 32 instances across all $\gamma$ values, corresponding to the first replication from each of the 32 instance families defined by $(I_{\mathrm{static}},J)$. Each trained model is evaluated on these testing instances under its corresponding value of $\gamma$. The lazy benchmark policy is evaluated under the same settings and on the same testing instances. The aggregated results are reported below.

\begin{table}[htbp]
    \centering
    \begin{threeparttable}
        \caption{\centering Performance of ML--CO and Lazy on 32 testing instances under different values of $\gamma$.}
        \label{tab:gamma_sensitivity}
        \small
        \setlength{\tabcolsep}{5pt}
        \begin{tabular}{cccccc}
            \toprule
            \multirow{3}{*}{\shortstack{Minimum battery \\level ratio ($\gamma$)}}
            & \multicolumn{3}{c}{Objective}
            & \multicolumn{2}{c}{Statistics of ML--CO} \\
            \cmidrule(lr){2-4}
            \cmidrule(lr){5-6}
            & \multirow{1}{*}{ML--CO obj.}
            & \multirow{1}{*}{Lazy obj.}
            & \multirow{1}{*}{\shortstack{Gain over Lazy (\%)}}
            & \multirow{1}{*}{\shortstack{Avg. vehicles}}
            & \multirow{1}{*}{\shortstack{Avg. charging-station visits}} \\
            \midrule
            0.1 & 2459.7 & 2681.3 & 8.0 & 19.0 & 0.9 \\
            0.4 & 2470.6 & 2738.3 & 9.8 & 19.5 & 5.5 \\
            0.7 & 2576.0 & 2840.0 & 9.3 & 26.8 & 19.7 \\
            \bottomrule
        \end{tabular}
    \end{threeparttable}
\end{table}

From Table \ref{tab:gamma_sensitivity}, we can see that a higher value of $\gamma$ leads to higher routing costs for both ML--CO and the Lazy policy. For ML--CO, the average number of charging-station visits increases from 0.9 at $\gamma=0.1$ to 5.5 at $\gamma=0.4$ and 19.7 at $\gamma=0.7$. The number of dispatched vehicles changes only slightly when $\gamma$ increases from 0.1 to 0.4, but increases substantially to 26.8 at $\gamma=0.7$. This shows that a moderate battery requirement can mainly be handled by additional charging time at visited charging stations. However, when $\gamma=0.7$, vehicles need to visit charging stations more often and more vehicles need to be dispatched. As a result, the objective value increases with $\gamma$. The increase is small from $\gamma=0.1$ to $0.4$, but becomes much larger at $\gamma=0.7$.

We also observe that ML--CO consistently outperforms the Lazy policy under all three settings, with gains ranging from 8.0\% to 9.8\%. This shows that the learned request-selection policy remains effective even when the battery requirement becomes stricter and more charging is needed.

\subsubsection{Effect of different weight vectors} \label{subsubsec:effect_of_weights}
To examine how the objective weights affect the resulting solutions, we consider different weight vectors $(w_1,w_2)$ for total travel time and total excess user ride time. In addition to the default setting $(1,0)$, we evaluate $(0.75,0.25)$ and $(0.5,0.5)$ to investigate the trade-off between operational efficiency and user ride quality. Similar to Sections \ref{subsubsec:different_epoch_length} and \ref{subsubsec:different_battery_levels}, we use the same set of training instances and generate training samples separately for each $(w_1,w_2)$ setting. We then train a separate ML--CO model for each weight vector. For testing, we use the same 32 instances across all settings, corresponding to the first replication from each instance family. Each trained model is evaluated under its corresponding weight vector. The Lazy benchmark is evaluated under the same settings and on the same testing instances. The aggregated results are reported below.

\begin{table}[htbp]
    \centering
    \begin{threeparttable}
        \caption{\centering Performance under different objective weight vectors.}
        \label{tab:weight_sensitivity}
        \small
        \setlength{\tabcolsep}{6pt}
        \begin{tabular}{cccccccc}
            \toprule
            & \multicolumn{3}{c}{Objective}
            & \multicolumn{4}{c}{Statistics of ML--CO} \\
            \cmidrule(lr){2-4}
            \cmidrule(lr){5-8}
            $(w_1,w_2)$
            & ML--CO
            & Lazy
            & \shortstack{Gap with\\Lazy}
            & \shortstack{Travel\\time}
            & \shortstack{Excess user\\ride time}
            & \shortstack{Avg.\\vehicles}
            & \shortstack{Avg. fragment\\length} \\
            \midrule
            $(1,0)$
            & 2455.4
            & 2681.3
            & 8.4\%
            & 2455.4
            & 1492.0
            & 18.9
            & 4.0 \\

            $(0.75,0.25)$
            & 2046.7
            & 2188.8
            & 6.5\%
            & 2427.7
            & 904.0
            & 20.5
            & 3.1 \\

            $(0.5,0.5)$
            & 1519.2
            & 1586.1
            & 4.2\%
            & 2592.2
            & 446.3
            & 22.7
            & 2.5 \\
            \bottomrule
        \end{tabular}
    \end{threeparttable}
\end{table}

As expected, increasing the weight on the total excess user ride time leads to a clear reduction in this component. At the same time, the performance gap between ML--CO and Lazy decreases from 8.4\% to 4.2\%. One possible explanation is that a higher weight on the total excess ride time encourages more direct pickup-and-delivery service and reduces ride-sharing opportunities, leaving less room for ML--CO to benefit from strategic request selection and postponement. 
This is supported by the average fragment length, which decreases from 4.0 under $(1,0)$ to 2.5 under $(0.5,0.5)$. Meanwhile, the average number of dispatched vehicles increases from 18.9 to 22.7, which further indicates a less consolidated service.

Interestingly, a moderate weight on excess user ride time, i.e., $(0.75,0.25)$, reduces both excess ride time and total travel time compared with $(1,0)$. This suggests that some consideration of user ride quality can improve route organization by avoiding inefficient detours. However, when this weight becomes larger, shorter fragments and more dispatched vehicles indicate reduced ride-sharing, causing total travel time to increase. Overall, the results show a clear trade-off between ride-sharing efficiency and user ride quality.

\section{Conclusion} \label{sec:conclusion}
In this work, we introduce a new variant of the dynamic E-ADARP based on decision epochs. In the Dyn-EADARP, incoming requests are collected and processed at a sequence of decision epochs, thereby decoupling request arrivals from dispatching decisions.
At each epoch, the service provider decides which available requests to dispatch and which to postpone, while jointly determining vehicle routes, schedules, and charging decisions. The unexecuted portions of previously planned routes can also be revised as new information becomes available. By making the timing of request dispatch an explicit decision, this setting allows the service provider to exploit future ride-sharing opportunities while adapting existing plans to newly arrived requests.

To solve the Dyn-EADARP, we develop a ML--CO policy within the COAML framework integrating a statistical model and a CO layer. The statistical model predicts a prize for each currently available request, while the downstream optimization layer solves a prize-collecting E-ADARP to jointly determine request selection and vehicle plans. To make this optimization layer computationally practical, we design a problem-specific heuristic that accounts for pickup--drop-off pairing, ride-sharing, vehicle charging, and revision of previously planned routes. The statistical model is trained with the objective of improving the quality of decisions produced by the optimization layer.

Computational experiments on 320 testing instances demonstrate the effectiveness of the proposed policy. The ML--CO policy improves the objective value by 8.7\% compared with the best benchmark policy and achieves an average gap of 4.8\% to the approximate anticipative lower bound. It also solves instances with nearly 500 requests efficiently, requiring about 6 seconds on average. The results provide several managerial insights. First, serving requests immediately is not always beneficial, as strategically postponing some requests can create better ride-sharing opportunities and reduce routing costs. Second, allowing the unexecuted portions of existing plans to be revised preserves operational flexibility and improves solution quality. Third, more frequent decision making does not necessarily lead to better performance, highlighting the importance of choosing an appropriate decision frequency. Additional experiments examine the effects of statistical-model architecture, training-set size, routing flexibility, decision frequency, battery requirements, and objective weights, and show that the main conclusions are robust across a broad range of settings.


Several directions could be considered for future research. First, the current framework can be extended to settings with more realistic request-arrival processes and heterogeneous service requirements. Second, future research can also consider a stochastic Dyn-EADARP with uncertain travel times or energy consumption.
Finally, it would be interesting to extend the proposed decision-focused framework to other dynamic ride-sharing and demand-responsive transportation problems in which operational decisions must be repeatedly updated as new information becomes available.

\bibliography{mybib.bib} 

\newpage
\appendix

\section{Fenchel-Young loss}
\label{app:FY_loss}
Recall that the objective of the CO layer can be written as
\begin{equation}
G_e(\boldsymbol{y}_e;\boldsymbol{\theta}_e)
=
\boldsymbol{\theta}_e^\top \boldsymbol{y}_e
-
C_e(\boldsymbol{y}_e),
\end{equation}
where $C_e(\boldsymbol{y}_e)$ is the routing cost and does not depend on $\boldsymbol{\theta}_e$. Therefore, we have
\begin{equation}
\label{eq:gradient_G_e}
\nabla_{\boldsymbol{\theta}_e}G_e(\boldsymbol{y}_e;\boldsymbol{\theta}_e) = \boldsymbol{y}_e
\end{equation}

Then, we define the perturbed function as

\begin{equation}
F_{\epsilon}(\boldsymbol{\theta}_e)
=
\mathbb{E}_{Z}
\left[
\max_{\boldsymbol{y}_e\in\mathcal{Y}(x_e)}
G_e
\left(
\boldsymbol{y}_e;
\boldsymbol{\theta}_e+\epsilon Z
\right)
\right].
\end{equation}
As $G_e(\boldsymbol{y}_e;\boldsymbol{\theta}_e)$ is affine in $\boldsymbol{\theta}_e$, $F_{\epsilon}(\boldsymbol{\theta}_e)$ is a convex function with $\boldsymbol{\theta}_e$. Following the perturbed-optimizer results of \cite{berthet2020learning}, we have the gradient of $F_{\epsilon}(\boldsymbol{\theta}_e)$ equals to:
\begin{equation}
\label{eq:gradient_F}
\nabla_{\boldsymbol{\theta}_e}
F_{\epsilon}(\boldsymbol{\theta}_e)
=
\mathbb{E}_{Z}
\left[
\boldsymbol{y}_e^*(Z)
\right],
\end{equation}
where $\boldsymbol{y}_e^*(Z)$ is defined in Equation \eqref{eq:perturbed_optimal_decision}.
Therefore, the perturbed loss in Equation \eqref{eq:perturbed_loss} can be written as 
\begin{equation}
\mathcal{L}_{\epsilon}
(\boldsymbol{\theta}_e,\bar{\boldsymbol{y}}_e)
=
F_{\epsilon}(\boldsymbol{\theta}_e)
-
G_e
\left(
\bar{\boldsymbol{y}}_e;
\boldsymbol{\theta}_e
\right).
\end{equation}

According to Equation \eqref{eq:gradient_G_e} and \eqref{eq:gradient_F}, we have:
\begin{equation}
\nabla_{\boldsymbol{\theta}_e}
\mathcal{L}_{\epsilon}
=
\mathbb{E}_{Z}
\left[
\boldsymbol{y}_e^*(Z)
\right]
-
\bar{\boldsymbol{y}}_e,
\end{equation}
which proves Proposition~\ref{prop:perturbed_gradient}.

\section{Mathematical formulation for the Anticipative Dyn-EADARP} \label{anticipative Dyn-EADARP}
Following \cite{bongiovanni2019electric}, we use binary variables $x_{ij}^k$ which are equal to $1$ if arc $(i,j)$ is used by vehicle $k$ and $0$ otherwise. Furthermore, we use continuous variables $T_i^k$ to store the beginning of service time at a node $i$ by vehicle $k$, $L_i^k$ for the load when leaving node $i$ with vehicle $k$, $B_i^k$ to keep track of the battery level at each node $i$ and vehicle $k$, $E_s^k$ for the time available for recharging at charging station $s$ and vehicle $k$, and $R_i$ to store the excess ride time of user request $i$. $\tau_i$ denotes the release time of request $i \in P$.

\begin{align} \label{EQsec3:objective1}
     & \min w_1 \sum\limits_{k \in K} \sum\limits_{(i,j) \in A}t_{i,j}x_{i,j}^k + w_2 \sum\limits_{i \in P} R_i
\end{align}
subject to:
\begin{align} 
\label{EQsec3:start}
 & \sum\limits_{j \in P\cup S \cup F} x_{o^k,j}^{k} = 1, 
 && \forall k \in K
 \\
\label{EQsec3:end}
& \sum\limits_{i \in D\cup S \cup \{o^k\}}\sum\limits_{j \in F} x_{i,j}^{k} = 1, 
&& \forall k \in K
\\ 
\label{EQsec3:inout}
& \sum\limits_{j \in V, j \neq i} x_{i,j}^{k} - \sum\limits_{j \in V, j \neq i} x_{j,i}^{k} = 0, 
&&\forall k \in K, i \in N \cup S
\\
\label{EQsec3:visitPickup}
& \sum\limits_{k \in K} \sum\limits_{j \in N, j\neq i} x_{i,j}^{k} = 1, 
&& \forall i \in P
\\
\label{EQsec3:delivery}
& \sum\limits_{j \in N, j \neq i} x_{i,j}^{k} - \sum\limits_{j \in N, j \neq n+i} x_{j,n+i}^{k} = 0, 
&& \forall k \in K, i \in P
\\ 
\label{EQsec3:precedence}
& T_i^k + s_i + t_{i,n+i} \leqslant T_{n+i}^k, 
&&\forall k \in K, i \in P
\\ 
\label{EQsec3:tw}
& e_i \leqslant T_i^k \leqslant l_i, 
&& \forall k \in K, i \in V
\\ 
\label{EQsec3:release time}
&  T_i^k \geqslant \tau_i, 
&& \forall k \in K, i \in P
\\ 
\label{EQsec3:time}
& T_i^k + t_{i,j} + s_i - M_{i,j}(1-x_{i,j}^k) \leqslant T_j^k, 
&& \forall k \in K, (i,j) \in A 
\\ 
\label{EQsec3:ride}
& T_{n+i}^k - T_{i}^k - s_i \leqslant m_i, 
&& \forall k \in K, i \in P
\\ 
\label{EQsec3:excess}
& R_i \geqslant T_{n+i}^k - T_{i}^k - s_i - t_{i,n+i}, 
&& \forall k \in K, i \in P
\\ 
\label{EQsec3:load1}
& L_i^k + q_j  - G_{i,j}^k(1-x_{i,j}^k) \leqslant L_j^k, 
&& \forall k \in K, (i,j) \in A 
\\ 
\label{EQsec3:load2}
& L_i^k + q_j + G_{i,j}^k(1-x_{i,j}^k) \geqslant L_j^k, 
&& \forall k \in K, (i,j) \in A 
\\ 
\label{EQsec3:loadLB}
& L_i^k \geqslant \max\{0,q_i\}, 
&& \forall k \in K, i \in N
\\ 
\label{EQsec3:loadUB}
& L_i^k \leqslant \min\{C,C+q_i\}, 
&& \forall k \in K, i \in N
\\ 
\label{EQsec3:loadZero}
& L_i^k = 0, 
&& \forall k \in K, i \in \{o^k\} \cup F \cup S
\\ 
\label{EQsec3:SoCZero}
& B_{o^k}^k = B_0^k, 
&& \forall k \in K
\\ 
\label{EQsec3:SoC1}
& B_j^k \leqslant B_i^k - b_{i,j} + Q(1-x_{i,j}^k), 
&& \forall k \in K, i \in V \setminus S, j \in V  \setminus \{o^k\}, i \neq j
\\ 
\label{EQsec3:SoC2}
& B_j^k \geqslant B_i^k - b_{i,j} - Q(1-x_{i,j}^k), 
&& \forall k \in K, i \in V \setminus S, j \in V  \setminus \{o^k\}, i \neq j
\\
\label{EQsec3:SoC3}
& B_j^k \leqslant B_s^k + \alpha E_s^k - b_{s,j} + Q(1-x_{s,j}^k), 
&& \forall k \in K, s \in S, j \in P \cup F \cup S, s \neq j
\\ 
\label{EQsec3:SoC4}
& B_j^k \geqslant B_s^k + \alpha E_s^k - b_{s,j} - Q(1-x_{s,j}^k), 
&& \forall k \in K, s \in S, j \in P \cup F \cup S, s \neq j
\\ 
\label{EQsec3:SoC5}
& Q \geqslant B_s^k + \alpha E_s^k, 
&& \forall k \in K, s \in S
\\ 
\label{EQsec3:SoC6}
& B_i^k \geqslant \gamma Q, 
&& \forall k \in K, i \in F
\\
\label{EQsec3:timeCharge1}
 &   E_s^k \leqslant T_i^k - t_{s,i} - T_s^k + \tilde M_{s,i}^k\left(1-x_{s,i}^{k}\right), 
 && \forall s \in S, i \in P\cup{S}\cup{F}, k\in K, i \ne s
\\ 
\label{EQsec3:timeCharge2}
&    E_s^k \geqslant T_i^k - t_{s,i} - T_s^k - \tilde M_{s,i}^k\left(1-x_{s,i}^{k}\right), 
&& \forall s \in S, i \in P\cup{S}\cup{F}, k\in K, i \ne s
\\ 
\label{sec3: cons26}
&    x_{i,j}^k \in \{0,1\}, 
&& \forall k \in K, (i,j) \in A 
\\ 
\label{sec3: cons27}
&    B_i^k \geqslant 0, 
&& \forall k \in K, i \in V
\\ 
\label{sec3: cons28}
&    E_s^k \geqslant 0, 
&& \forall k \in K, s \in S
\end{align}

Objective function~\eqref{EQsec3:objective1} minimizes a weighted sum of total travel time and total excess user ride time. Constraints~\eqref{EQsec3:start} and ~\eqref{EQsec3:end} make sure that each vehicle leaves from its depot and returns to one of the end depots. Flow conservation is taken care of by~\eqref{EQsec3:inout}.   Constraints~\eqref{EQsec3:visitPickup}--\eqref{EQsec3:precedence} guarantee that each pickup node is visited exactly once, that the same vehicle that visits pickup node $i$ also visits drop-off node $n+i$, and in the correct order. 
Constraints~\eqref{EQsec3:tw} -- \eqref{EQsec3:time} take care of time windows and of correctly setting the beginning of service time variables. Specifically, Constraints~\eqref{EQsec3:release time} guarantee that the service start time is not earlier than the release time of a request.
Maximum user ride time restrictions and the correct computation of the excess user ride times are handled by \eqref{EQsec3:ride} and \eqref{EQsec3:excess}, respectively. Constraints \eqref{EQsec3:load1} -- \eqref{EQsec3:loadZero} make sure that the vehicle capacity (seats available) is respected at all times. The remaining constraints enforce the battery-related requirements. Constraints~\eqref{EQsec3:SoCZero} initialize the state of charge (SoC) of each vehicle at the origin depot. Constraints~\eqref{EQsec3:SoC1} and~\eqref{EQsec3:SoC2} then update the SoC at pickup and drop-off nodes, while constraints~\eqref{EQsec3:SoC3} and~\eqref{EQsec3:SoC4} account for battery recharging at charging stations. The latter work together with constraints~\eqref{EQsec3:timeCharge1} and~\eqref{EQsec3:timeCharge2}, which determine the time available for recharging during each charging-station visit. Finally, constraints \eqref{EQsec3:SoC5} and \eqref{EQsec3:SoC6} make sure that the battery capacity and the minimum battery level at the depot are respected. Note that another constraint is added in \cite{bongiovanni2019electric} to make sure each charging station is visited at most once:
\begin{align}
    \label{EQsec3:chargevisit}
& \sum\limits_{k \in K} \sum\limits_{i \in D\cup S \cup \{o^k\}} x_{i,j}^{k} \leqslant 1, 
&&\forall j \in F \cup S
\end{align}
We therefore allow multiple visits to a recharging station by replicating the set of recharging stations, as in \cite{bongiovanni2019electric} and \cite{su2024branch}

Tight big-M values $M_{i,j} = l_i + t_{ij} + s_i$ for each arc $(i,j) \in A$ are defined to make sure that Constraints \eqref{EQsec3:time} are only binding when $x_{ij}^k = 1$.
Similarly, we take $G_{i,j}^k = \max\{C^k,C^k+q_j\}$ for Constraints ~\eqref{EQsec3:load1} and ~\eqref{EQsec3:load2} and $\tilde M_{s,i}^k = T$ for Constraints ~\eqref{EQsec3:timeCharge1} and ~\eqref{EQsec3:timeCharge2}.
A feasible solution of the anticipative Dyn-EADARP is a set of $|K|$ routes with all constraints listed above being satisfied.




\section{Dynamic instance generation}
\label{app:dynamic_instances}

\begin{itemize}
    \item \textbf{Sample locations and time windows:} We first independently sample two locations from all pickup and drop-off locations in the original static instance and use them as the pickup and drop-off locations of the new request. We then randomly sample a pair of pickup and drop-off time windows from the original static instance. Requests with ``too close'' or ``too far'' locations are discarded. Specifically, a pair is considered ``too close'' if the distance between the sampled locations is smaller than $\min_{i\in P} d_{i,n+i}$ in the original instance, and ``too far'' if the corresponding travel time exceeds the maximum user ride time.
    
    \item \textbf{Adjust sampled time windows of a request:} Let request $r$ denote the original request associated with the sampled time-window pair, and let $i$ denote the newly generated request. We first adjust the drop-off time window according to the change in travel time from pickup to drop-off. Specifically, we define $\Delta_i^{\mathrm{D}}$ as the difference between $t_{i,n+i}$ and $t_{r,n+r}$:
    \[
        \Delta_i^{\mathrm{D}}
        =
        t_{i,n+i} - t_{r,n+r},
    \]
    and set the drop-off time window of the new request as
    \[
        [e_{n+i}, l_{n+i}]
        =
        [e_{n+r} + \Delta_i^{\mathrm{D}},
         \; l_{n+r} + \Delta_i^{\mathrm{D}}].
    \]
    We then adjust the pickup time window according to the change in travel
    time from the origin depot to the pickup location. Let $o$ denote the
    origin depot. We define
    \[
        \Delta_i^{\mathrm{P}}
        =
        t_{o,i} - t_{o,r},
    \]
    and set the pickup time window as
    \[
        [e_i, l_i]
        =
        [e_r + \Delta_i^{\mathrm{P}},
         \; l_r + \Delta_i^{\mathrm{P}}].
    \]
    \item \textbf{Check request feasibility:} The generated requests are checked for release time feasibility and requests that are incompatible with release time constraints are discarded. For a request $i$ generated at epoch $e$, it is kept if and only if
    \begin{align}
    \tau_e + \min_{o\in O} t_{o,i} < l_i,
    \label{eq:release_time}
    \end{align}
    where $\tau_e + \min_{o\in O} t_{o,i}$ represents the earliest possible service start time at pickup node $i$ if a vehicle departs from the closest origin depot $o$ at the release time $\tau_e$.
    For requests generated in the last epoch, we additionally require that the request can be completed within the planning horizon. Specifically, we calculate the earliest completion time $\rho_{i,e}$, defined as 
    \begin{align}
    \rho_{i,e} = \tau_{e} + \min_{o\in O} t_{o,i} + s_i + t_{i,n+i} + s_{n+i}
    + \min_{f\in F} t_{n+i,f},
    \label{eq:completion_time}
    \end{align}
    and make sure
    \begin{align}
        \rho_{i,e}\leq T.
        \label{eq:planning_horizon}
    \end{align}
    Constraint \eqref{eq:planning_horizon} ensures that, even if the vehicle starts from the depot at the request release time $\tau_e$, it can serve the pickup and drop-off nodes and reach a destination depot before the end of the planning horizon.
    \item \textbf{Time-window tightening:} After generating the requests, their pickup and drop-off time windows are
further tightened according to the release-time, maximum ride-time,
precedence, and planning-horizon constraints:
\begin{itemize}
    \item For the pickup node $i$ of a request released at epoch $e$,
    \[
    \begin{aligned}
        e_i
        &\leftarrow
        \max\left\{
            e_i,\,
            e_{n+i}-m_i-s_i,\,
            \tau_e
        \right\},\\
        l_i
        &\leftarrow
        \min\left\{
            l_i,\,
            l_{n+i}-t_{i,n+i}-s_i
        \right\}.
    \end{aligned}
    \]

    \item For the corresponding drop-off node $n+i$,
    \[
    \begin{aligned}
        e_{n+i}
        &\leftarrow
        \max\left\{
            e_{n+i},\,
            e_i+t_{i,n+i}+s_i
        \right\},\\
        l_{n+i}
        &\leftarrow
        \min\left\{
            l_{n+i},\,
            l_i+m_i+s_i,\,
            T-s_{n+i}-\
            min_{f\in F}t_{n+i,f}
        \right\}.
    \end{aligned}
    \]

\end{itemize}
\end{itemize}

\newpage

\section{Detailed results}
\label{app:detailed_results}
Table~\ref{tab:expanded-objective-means-warmstart} reports the results by instance family. 
For each family, we report the mean objective value, the relative gap with respect to Ref, and the mean number of dispatched vehicles for each policy. Specifically, the mean objective and dispatched vehicle numbers for Ref are also reported.

\begingroup
\scriptsize
\setlength{\tabcolsep}{1.25pt}
\begin{longtable}{@{}l*{18}{c}@{}}
\caption{Performance on each instance family.}\label{tab:expanded-objective-means-warmstart}\\
\toprule
Instance & Avg. requests & \multicolumn{2}{c}{Anticipative Ref} & \multicolumn{3}{c}{Greedy} & \multicolumn{3}{c}{Lazy} & \multicolumn{3}{c}{Random} & \multicolumn{3}{c}{Rolling horizon} & \multicolumn{3}{c}{Linear ML--CO} \\
\cmidrule(lr){3-4}\cmidrule(lr){5-7}\cmidrule(lr){8-10}\cmidrule(lr){11-13}\cmidrule(lr){14-16}\cmidrule(lr){17-19}
 & & Obj. & Veh. & Obj. & Gap (\%) & Veh. & Obj. & Gap (\%) & Veh. & Obj. & Gap (\%) & Veh. & Obj. & Gap (\%) & Veh. & Obj. & Gap (\%) & Veh. \\
\midrule
\endfirsthead
\toprule
Instance & Avg. requests & \multicolumn{2}{c}{Anticipative target} & \multicolumn{3}{c}{Greedy} & \multicolumn{3}{c}{Lazy} & \multicolumn{3}{c}{Random} & \multicolumn{3}{c}{Rolling horizon} & \multicolumn{3}{c}{Linear ML--CO} \\
\cmidrule(lr){3-4}\cmidrule(lr){5-7}\cmidrule(lr){8-10}\cmidrule(lr){11-13}\cmidrule(lr){14-16}\cmidrule(lr){17-19}
 & & Obj. & Veh. & Obj. & Gap (\%) & Veh. & Obj. & Gap (\%) & Veh. & Obj. & Gap (\%) & Veh. & Obj. & Gap (\%) & Veh. & Obj. & Gap (\%) & Veh. \\
\midrule
\endhead
\midrule
\multicolumn{19}{r}{Continued on next page}\\
\endfoot
\bottomrule
\endlastfoot

a3-36-J18 & 95.9 & 1,293.0 & 11.2 & 1,604.5 & +24.1 & 10.3 & 1,555.3 & +20.3 & 9.3 & 1,579.6 & +22.2 & 11.2 & 1,606.5 & +24.3 & 10.5 & \textbf{1,400.5} & +8.3 & 12.7 \\
a3-36-J27 & 138.4 & 1,871.1 & 16.4 & 2,214.7 & +18.4 & 14.4 & 2,097.4 & +12.1 & 11.8 & 2,204.6 & +17.8 & 14.9 & 2,198.4 & +17.5 & 14.9 & \textbf{1,915.6} & +2.4 & 16.9 \\
a4-32-J16 & 60.7 & 817.1 & 8.4 & 1,020.8 & +24.9 & 8.5 & 959.4 & +17.4 & 7.6 & 1,015.3 & +24.3 & 9.3 & 1,008.4 & +23.4 & 8.5 & \textbf{906.5} & +10.9 & 10.0 \\
a4-32-J24 & 89.5 & 1,166.2 & 11.4 & 1,426.2 & +22.3 & 11.6 & 1,334.8 & +14.5 & 9.6 & 1,421.7 & +21.9 & 12.0 & 1,398.4 & +19.9 & 10.8 & \textbf{1,264.8} & +8.5 & 12.7 \\
a4-40-J20 & 91.8 & 1,234.4 & 11.2 & 1,514.5 & +22.7 & 10.5 & 1,431.3 & +16.0 & 9.5 & 1,495.6 & +21.2 & 11.4 & 1,493.2 & +21.0 & 10.4 & \textbf{1,305.7} & +5.8 & 11.9 \\
a4-40-J30 & 135.2 & 1,766.7 & 14.4 & 2,160.7 & +22.3 & 14.6 & 2,059.5 & +16.6 & 13.0 & 2,137.9 & +21.0 & 15.7 & 2,163.8 & +22.5 & 14.8 & \textbf{1,891.2} & +7.1 & 15.2 \\
a4-48-J24 & 135.4 & 1,673.6 & 14.4 & 2,048.6 & +22.4 & 14.6 & 1,934.3 & +15.6 & 12.7 & 1,992.1 & +19.0 & 14.6 & 2,012.2 & +20.2 & 14.6 & \textbf{1,787.5} & +6.8 & 15.2 \\
a4-48-J36 & 203.9 & 2,409.1 & 18.7 & 2,874.4 & +19.3 & 19.7 & 2,699.6 & +12.1 & 18.9 & 2,888.5 & +19.9 & 20.9 & 2,819.9 & +17.1 & 18.8 & \textbf{2,466.2} & +2.4 & 18.9 \\
a5-40-J20 & 73.0 & 919.2 & 9.3 & 1,173.5 & +27.7 & 9.3 & 1,107.0 & +20.4 & 8.6 & 1,148.6 & +25.0 & 10.0 & 1,144.0 & +24.5 & 10.0 & \textbf{1,023.9} & +11.4 & 11.0 \\
a5-40-J30 & 111.2 & 1,381.3 & 12.3 & 1,707.1 & +23.6 & 12.9 & 1,649.8 & +19.4 & 10.7 & 1,695.5 & +22.8 & 13.2 & 1,704.9 & +23.4 & 12.4 & \textbf{1,521.6} & +10.2 & 13.5 \\
a5-50-J25 & 142.8 & 1,709.0 & 15.2 & 2,138.2 & +25.1 & 15.6 & 2,013.7 & +17.8 & 13.7 & 2,133.1 & +24.8 & 15.7 & 2,090.1 & +22.3 & 15.2 & \textbf{1,841.0} & +7.7 & 16.1 \\
a5-50-J38 & 205.3 & 2,449.5 & 20.9 & 2,958.2 & +20.8 & 20.7 & 2,808.7 & +14.7 & 17.8 & 2,926.3 & +19.5 & 20.6 & 2,897.4 & +18.3 & 19.5 & \textbf{2,529.8} & +3.3 & 21.2 \\
a5-60-J30 & 188.7 & 2,353.0 & 18.7 & 2,905.0 & +23.5 & 20.2 & 2,673.5 & +13.6 & 18.3 & 2,805.6 & +19.2 & 19.9 & 2,807.3 & +19.3 & 20.0 & \textbf{2,429.6} & +3.3 & 18.8 \\
a5-60-J45 & 280.5 & 3,419.9 & 23.8 & 4,098.9 & +19.9 & 27.0 & 3,762.9 & +10.0 & 25.9 & 4,039.2 & +18.1 & 27.2 & 4,054.4 & +18.6 & 26.1 & \textbf{3,470.6} & +1.5 & 23.8 \\
a6-48-J24 & 104.9 & 1,259.3 & 12.7 & 1,535.4 & +21.9 & 12.0 & 1,477.8 & +17.4 & 10.9 & 1,545.1 & +22.7 & 12.5 & 1,526.7 & +21.2 & 12.0 & \textbf{1,354.1} & +7.5 & 14.9 \\
a6-48-J36 & 157.8 & 1,877.9 & 18.5 & 2,250.3 & +19.8 & 16.1 & 2,115.8 & +12.7 & 15.1 & 2,239.9 & +19.3 & 18.2 & 2,201.1 & +17.2 & 15.8 & \textbf{1,964.3} & +4.6 & 19.3 \\
a6-60-J30 & 136.2 & 1,732.2 & 14.5 & 2,168.1 & +25.2 & 16.0 & 2,015.5 & +16.4 & 14.0 & 2,082.3 & +20.2 & 15.8 & 2,134.8 & +23.3 & 16.5 & \textbf{1,835.9} & +6.0 & 14.5 \\
a6-60-J45 & 199.1 & 2,429.3 & 18.3 & 2,942.5 & +21.1 & 20.6 & 2,726.7 & +12.2 & 18.6 & 2,888.7 & +18.9 & 20.0 & 2,880.6 & +18.6 & 20.1 & \textbf{2,531.9} & +4.2 & 18.3 \\
a6-72-J36 & 227.6 & 2,743.2 & 19.4 & 3,299.2 & +20.3 & 19.5 & 3,049.9 & +11.2 & 17.9 & 3,259.3 & +18.8 & 20.1 & 3,287.1 & +19.8 & 19.1 & \textbf{2,774.6} & +1.1 & 19.4 \\
a6-72-J54 & 342.1 & 4,010.0 & 24.6 & 4,736.8 & +18.1 & 27.6 & 4,337.9 & +8.2 & 24.5 & 4,643.9 & +15.8 & 28.3 & 4,659.0 & +16.2 & 27.0 & \textbf{4,063.8} & +1.3 & 24.6 \\
a7-56-J28 & 112.5 & 1,471.2 & 13.5 & 1,825.1 & +24.1 & 13.4 & 1,761.6 & +19.7 & 12.4 & 1,815.8 & +23.4 & 14.4 & 1,810.3 & +23.1 & 13.8 & \textbf{1,588.9} & +8.0 & 14.4 \\
a7-56-J42 & 165.9 & 2,141.4 & 17.9 & 2,555.7 & +19.4 & 17.8 & 2,470.7 & +15.4 & 16.2 & 2,528.9 & +18.1 & 18.5 & 2,545.1 & +18.9 & 17.7 & \textbf{2,241.4} & +4.7 & 18.4 \\
a7-70-J35 & 171.4 & 2,166.3 & 22.0 & 2,643.8 & +22.0 & 22.3 & 2,505.0 & +15.6 & 24.8 & 2,607.6 & +20.4 & 22.8 & 2,602.6 & +20.1 & 23.0 & \textbf{2,312.7} & +6.8 & 24.2 \\
a7-70-J53 & 258.1 & 3,220.9 & 30.1 & 3,818.7 & +18.6 & 28.8 & 3,555.3 & +10.4 & 32.6 & 3,786.8 & +17.6 & 31.3 & 3,799.8 & +18.0 & 29.8 & \textbf{3,299.5} & +2.4 & 30.9 \\
a7-84-J42 & 244.8 & 3,007.3 & 20.1 & 3,550.4 & +18.1 & 20.2 & 3,319.5 & +10.4 & 18.3 & 3,505.0 & +16.6 & 20.9 & 3,545.3 & +17.9 & 20.6 & \textbf{3,059.7} & +1.7 & 20.3 \\
a7-84-J63 & 367.9 & 4,427.9 & 28.6 & 5,188.1 & +17.2 & 27.5 & 4,842.8 & +9.4 & 26.5 & 5,068.3 & +14.5 & 29.2 & 5,151.0 & +16.3 & 27.5 & \textbf{4,448.4} & +0.5 & 28.7 \\
a8-64-J32 & 126.6 & 1,528.4 & 13.5 & 1,891.5 & +23.8 & 12.9 & 1,775.7 & +16.2 & 11.1 & 1,851.6 & +21.2 & 13.6 & 1,851.8 & +21.2 & 13.1 & \textbf{1,599.5} & +4.7 & 14.0 \\
a8-64-J48 & 189.0 & 2,255.1 & 18.9 & 2,634.3 & +16.8 & 17.6 & 2,544.6 & +12.8 & 16.0 & 2,602.1 & +15.4 & 17.7 & 2,650.2 & +17.5 & 16.8 & \textbf{2,304.8} & +2.2 & 19.0 \\
a8-80-J40 & 205.9 & 2,392.1 & 18.8 & 2,888.9 & +20.8 & 18.8 & 2,699.4 & +12.8 & 17.7 & 2,864.5 & +19.8 & 19.5 & 2,862.1 & +19.6 & 19.0 & \textbf{2,447.9} & +2.3 & 18.8 \\
a8-80-J60 & 310.3 & 3,530.1 & 25.6 & 4,341.1 & +23.0 & 27.9 & 4,014.7 & +13.7 & 27.2 & 4,184.8 & +18.6 & 27.9 & 4,222.5 & +19.6 & 27.2 & \textbf{3,636.4} & +3.0 & 25.8 \\
a8-96-J48 & 301.5 & 3,629.4 & 23.8 & 4,401.9 & +21.3 & 23.1 & 4,177.6 & +15.1 & 22.6 & 4,362.6 & +20.2 & 24.4 & 4,329.2 & +19.3 & 23.2 & \textbf{3,641.3} & +0.3 & 23.8 \\
a8-96-J72 & 452.8 & 5,236.0 & 30.2 & 6,315.0 & +20.6 & 32.3 & 5,897.6 & +12.6 & 31.5 & 6,261.7 & +19.6 & 32.8 & 6,268.8 & +19.7 & 32.1 & \textbf{5,255.0} & +0.4 & 30.2 \\
\midrule
\textbf{Average} & 188.3 & 2,297.5 & 18.0 & 2,776.0 & +21.5 & 18.3 & 2,605.5 & +14.6 & 17.0 & 2,737.0 & +19.9 & 18.9 & 2,741.5 & +20.0 & 18.1 & \textbf{2,378.6} & +4.8 & 18.7 \\
\end{longtable}
\endgroup

 \newpage
\section{Implementation details for the effect of varying epoch lengths}
\label{app:epoch_instance_generation}
The experiments are organized as follows:

\begin{itemize}
\item \textbf{Generate dynamic instances for different values of $\eta$:}
For each $\eta\in\{10,20,30\}$, we construct a training set $M_{\eta}$ and a testing set $N_{\eta}$ from the original 60-minute sets $M_{60}$ and $N_{60}$, respectively. Here, $N_{60}$ contains 32 large testing instances, obtained by selecting the first replication from each of the 32 instance families defined by $(I_{static},J)$.

To isolate the effect of more frequent decision-making, we keep all request release times unchanged and only refine the decision epochs according to $\eta$. Specifically, each original 60-minute epoch is divided into $60/\eta$ shorter epochs. For example, when $\eta=30$, the original epoch $[0,60)$ is divided into $[0,30)$ and $[30,60)$, while all requests originally released at time 0 remain released at time 0. Hence, no additional requests are released at the newly inserted decision epochs, but the service provider can re-optimize the current plan more frequently. When $\eta=60$, the construction reduces to the original sets $M_{60}$ and $N_{60}$.

\item \textbf{Generate training samples and train new models:}
For each $\eta\in\{10,20,30\}$, we generate training samples from $M_{\eta}$ following the procedure described in Section~\ref{subsec:anticipative_target} and train a corresponding ML--CO model. With shorter epochs, a request selected at one decision epoch may only be picked up in a subsequent epoch. In this case, its training label is associated with the epoch in which the dispatching decision is made rather than the epoch in which the pickup is physically performed. 

\item \textbf{Evaluate the trained models:}
We evaluate the models trained under different epoch lengths. Each model trained on $M_{\eta}$, $\eta\in\{10,20,30,60\}$, is evaluated on testing sets with the same epoch lengths. 
\end{itemize}

\end{document}